\documentclass[reqno,11pt]{amsart}
\usepackage{amscd,amssymb,verbatim,array}
\usepackage{hyperref}

\numberwithin{equation}{section} \theoremstyle{plain}

\newcommand{\Complex}{\mathbb C}
\newcommand{\Real}{\mathbb R}
\newcommand{\N}{\mathbb N}
\newcommand{\ddbar}{\overline\partial}
\newcommand{\pr}{\partial}
\newcommand{\ol}{\overline}

\newcommand{\set}[1]{\left\{#1\right\}}
\newcommand{\To}{\rightarrow}

\newtheorem{theorem}{Theorem}[section]
\newtheorem{lemma}[theorem]{Lemma}
\newtheorem{proposition}[theorem]{Proposition}
\newtheorem{corollary}[theorem]{Corollary}
\newtheorem{definition}[theorem]{Definition}
\newtheorem{ass}[theorem]{Assumption}

\theoremstyle{definition}

\theoremstyle{remark}
\newtheorem{remark}[theorem]{Remark}

\numberwithin{equation}{section}

\newcommand{\abs}[1]{\lvert#1\rvert}

\begin{document}

\title[Asymptotics of G-equivariant Szeg\H{o} kernels]
{Asymptotics of G-equivariant Szeg\H{o} kernels}

\author{Rung-Tzung Huang}
\thanks{The first author was supported by Taiwan Ministry of Science and Technology projects 107-2115-M-008-007-MY2 and 109-2115-M-008-007-MY2.}
\address{Department of Mathematics, National Central University, Chung-Li, Taoyuan 32001, Taiwan}

\email{rthuang@math.ncu.edu.tw}

\author{Guokuan Shao}

\address{School of Mathematics (Zhuhai), Sun Yat-sen University, Zhuhai 519082, Guangdong, China}

\email{shaogk@mail.sysu.edu.cn}


\begin{abstract}
Let $(X, T^{1,0}X)$ be a compact connected orientable CR manifold of dimension $2n+1$ with non-degenerate Levi curvature. Assume that $X$ admits a connected compact Lie group $G$ action. Under certain natural assumptions about the group $G$ action, we define $G$-equivariant Szeg\H{o} kernels and establish the associated Boutet de Monvel-Sj\"ostrand type theorems. When $X$ admits also a transversal CR $S^1$ action, we study the asymptotics of Fourier components of $G$-equivariant Szeg\H{o} kernels with respect to the $S^1$ action.
\end{abstract}

\maketitle


\section{Introduction and statement of the main results}\label{s-gue170124}

Let $(X, T^{1,0}X)$ be a CR manifold of dimension $2n+1$, $n\geq1$.
Let $\Box^{(q)}_b$ be the Kohn Lalpacian acting on $(0,q)$ forms. 
The orthogonal projection $S^{(q)}:L^2_{(0,q)}(X)\To {\rm Ker\,}\Box^{(q)}_b$ onto ${\rm Ker\,}\Box^{(q)}_b$ is called the Szeg\H{o} projection. The Szeg\H{o} kernel is its distribution kernel $S^{(q)}(x,y)$. The study of the Szeg\H{o} projection and kernels is a classical subject in several complex variables and CR geometry.
When $X$ is the boundary of a strictly pseudoconvex domain,
Boutet de Monvel-Sj\"ostrand~\cite{BouSj76} showed that $S^{(0)}(x,y)$
is a complex Fourier integral operator. 

The Boutet de Monvel-Sj\"ostrand theorem had a profound impact
in many research areas, especially through  \cite{BG81}: several complex variables, 
symplectic and contact geometry, geometric quantization, K\"ahler geometry, semiclassical analysis,
quantum chaos, etc. cf.\ \cite{Ca99, Engl:02,Gu89, MM06,MM,MM08a,Ma10,ShZ99,Zelditch98}. Recently, Hsiao-Huang \cite{HH} obtained $G$-invariant Boutet de Monvel-Sj\"ostrand type theorems and Hsiao-Ma-Marinescu~\cite{HMM} established geometric quantization on CR manifolds by using  $G$-invariant Szeg\H{o} kernels asymptotic expansions. 

In this paper, we study $G$-equivariant Szeg\H{o} kernels with respect to all equivalent classes of irreducible unitary representations of $G$. We establish $G$-equivariant Boutet de Monvel-Sj\"ostrand type theorems. When $X$ admits also a transversal CR $S^1$ action,  we derive the asymptotic expansion of Fourier components of $G$-equivariant Szeg\H{o} kernels with respect to the $S^1$ action.

We now formulate the main results. We refer to Section~\ref{s:prelim} for some notations and terminology used here. Let $(X, T^{1,0}X)$ be a compact connected orientable CR manifold of dimension $2n+1$, $n\geq1$, where $T^{1,0}X$ denotes the CR structure of $X$. Fix a global non-vanishing real $1$-form $\omega_0\in C^\infty(X,T^*X)$ such that $\langle\,\omega_0\,,\,u\,\rangle=0$, for every $u\in T^{1,0}X\oplus T^{0,1}X$. The Levi form of $X$ at $x\in X$ is the Hermitian quadratic form on $T^{1,0}_xX$ given by $$\mathcal{L}_x(U,\ol V)=-\frac{1}{2i}\langle\,d\omega_0(x)\,,\,U\wedge\ol V\,\rangle, \ \text{for all} \ U, V\in T^{1,0}_xX.$$ 
In this paper, we assume that 

\begin{ass}\label{a-gue170123}
	The Levi form is non-degenerate of constant signature $(n_-,n_+)$ on $X$. That is, the Levi form  has exactly $n_-$ negative and $n_+$ positive eigenvalues at each point of $X$, where $n_-+n_+=n$. 
\end{ass}

Let $HX=\set{{\rm Re\,}u;\, u\in T^{1,0}X}$ and let $J:HX\To HX$ be the complex structure map given by $J(u+\ol u)=iu-i\ol u$, for every $u\in T^{1,0}X$. 
In this paper, we assume that $X$ admits a $d$-dimensional connected compact Lie group $G$ action. Let $\mathfrak{g}$ denote the Lie algebra of $G$. For any $\xi \in \mathfrak{g}$, we write $\xi_X$ to denote the vector field on $X$ induced by $\xi$. That is, $(\xi_X u)(x)=\frac{\partial}{\partial t}\left(u(\exp(t\xi)\circ x)\right)|_{t=0}$, for any $u\in C^\infty(X)$. Let $\underline{\mathfrak{g}}={\rm Span\,}(\xi_X;\, \xi\in\mathfrak{g})$.
We assume throughout that

\begin{ass}\label{a-gue170123I}
	The Lie group $G$ action is CR and preserves $\omega_0$ and $J$.
\end{ass}

We recall that the Lie group $G$ action preserves $\omega_0$ and $J$  means that 
$g^\ast\omega_0=\omega_0$ on $X$ and $g_\ast J=Jg_\ast$ on $HX$, for every $g\in G$, where $g^*$ and $g_*$ denote  the pull-back map and push-forward map of $G$, respectively. The $G$ action is CR means that for every $\xi_X\in\underline{\mathfrak{g}}$, 
	\begin{equation*}
		[\xi_X, C^{\infty}(X,T^{1,0}X)]\subset C^{\infty}(X,T^{1,0}X).
	\end{equation*}

\begin{definition}\label{d-gue170124}
	The moment map associated to the form $\omega_0$ is the map $\mu:X \to \mathfrak{g}^*$ such that, for all $x \in X$ and $\xi \in \mathfrak{g}$, we have 
	\begin{equation}\label{E:cmpm}
	\langle \mu(x), \xi \rangle = \omega_0(\xi_X(x)).
	\end{equation}
\end{definition}

We assume also that 

\begin{ass}\label{a-gue170123II}
	$0$ is a regular value of $\mu$ and $G$ acts locally free near $\mu^{-1}(0)$. 
\end{ass}

By Assumption~\ref{a-gue170123II}, $\mu^{-1}(0)$ is a $d$-codimensional orbifold of $X$. 
Note that if $G$ acts freely near $\mu^{-1}(0)$ and the Levi form is positive at $\mu^{-1}(0)$, it is known that 
 $\mu^{-1}(0)/G$ is a CR manifold with natural CR structure induced by $T^{1,0}X$ of dimension $2n-2d+1$, (see \cite{HH}). 


Let $R=\{R_1,R_2,...\}$ be the collection of all irreducible unitary representations of $G$, including only one representation from each equivalent class (see Section \ref{subs1}).
Write
\begin{equation}\label{e-9291}
\begin{split}
R_k:G&\to GL(\mathbb{C}^{d_k}), \ \ d_k<\infty,\\
g&\to (R_{k,j,l}(g))_{j,l=1}^{d_k},
\end{split}
\end{equation}
where $d_k$ is the dimension of the representation $R_k$.
Denote by $\chi_k(g):=\text{Tr}R_k(g)$ the trace of the matrix $R_k(g)$ (the character of $R_k$).
Let $u\in\Omega^{0,q}(X)$. For every $k = 1, 2, \cdots,$ we define 
\begin{equation}\label{e-9292}
u_k(x)=d_k\int_G (g^\star u)(x)\overline{\chi_k(g)}d\mu(g),
\end{equation}
where $d\mu(g)$ is the probability Haar measure on $G$.
We will show that $u=\sum_{k=0}^{\infty}u_k$.
By Assumption \ref{a-gue170123I}, 
$u_k\in \Omega^{0,q}(X)$ if $u\in \Omega^{0,q}(X)$.
Set
\begin{equation}\label{e-9293}
\Omega^{0,q}(X)_k:=\{u(x)\in \Omega^{0,q}(X)| u(x) = u_k(x)\}.
\end{equation}
Denote by $L^2_{(0,q)}(X)_k$ the completion of $\Omega^{0,q}(X)_k$ with respect to the inner product $(\,\cdot\,|\,\cdot\,)$. 
\begin{definition}\label{d-08}
The $G$-equivariant Szeg\H{o} projection is the orthogonal projection 
\[
S^{(q)}_k
:L^2_{(0,q)}(X)\To {\rm Ker\,}\Box^{(q)}_b\bigcap L^2_{(0,q)}(X)_k
\]
with respect to $(\,\cdot\,|\,\cdot\,)$. 
Denote by $S^{(q)}_k(x,y)\in D'(X\times X,T^{*0,q}X\boxtimes(T^{*0,q}X)^*)$ the distribution kernel of $S^{(q)}_k$.
\end{definition}
 The first main result of this work is the following 
\begin{theorem}\label{t-gue170124}
With the assumptions and notations above, suppose that $\Box^{(q)}_b : {\rm Dom\,}\Box^{(q)}_b\To L^2_{(0,q)}(X)$ has closed range. If $q\notin\set{n_-, n_+}$, then $S^{(q)}_k\equiv 0$ on $X$. 
	
Suppose  $q\in\set{n_-, n_+}$. Let $D$ be an open neighborhood of $X$ with $D\bigcap\mu^{-1}(0)=\emptyset$. Then, $S^{(q)}_k\equiv0$ on $D$. 

Let $p\in\mu^{-1}(0)$ and let $U$ be an open neighborhood of  $p$ and let $x=(x_1,\ldots,x_{2n+1})$ be local coordinates defined in $U$. 
Let $N_p=\{g\in G: g\circ p=p\}=\{g_1=e_0,g_2...,g_r\}$.
Then, there exist continuous operators $\hat S_{k,-}, \hat S_{k,+}:\Omega^{0,q}_c(U)\To\Omega^{0,q}(U)$
such that 
\begin{equation}\label{e-gue170108wrm}
S^{(q)}_k\equiv \hat S_{k,-}+\hat S_{k,+}\ \ \mbox{on $U$},
\end{equation}
and $\hat S_{k,-}(x,y)$, $\hat S_{k,+}(x,y)$ satisfy
\begin{equation}\label{e-gue170108wrIIm}
\begin{split}
&\hat S_{k,-}(x, y)\equiv\sum_{\alpha=1}^r\int^{\infty}_{0}e^{i\Phi_{k,-}(g_\alpha\circ x, y)t}a_{k,\alpha,-}(x, y, t)dt\ \ \mbox{on $U$},\\
&\hat S_{k,+}(x, y)\equiv\sum_{\alpha=1}^r\int^{\infty}_{0}e^{i\Phi_{k,+}(g_\alpha\circ x, y)t}a_{k,\alpha,+}(x, y, t)dt\ \ \mbox{on $U$},
\end{split}
\end{equation}
with 
\begin{equation}  \label{e-gue170108wrIIIm}\begin{split}
&a_{k,\alpha,+}(x, y, t), a_{k,\alpha,-}(x, y, t)\in S^{n-\frac{d}{2}}_{{\rm cl\,}}(U\times U\times\mathbb{R}_+,T^{*0,q}X\boxtimes(T^{*0,q}X)^*), \\
&a_{k,\alpha,-}(x,y,t)=0\ \ \mbox{if $q\neq n_-$},\ \ a_{k,\alpha,+}(x,y,t)=0\ \ \mbox{if $q\neq n_+$},\\
&a^0_{k,\alpha,-}(x,x)\neq0,\ \ \forall x\in U, \ \ \mbox{if $q=n_-$},\ \ a^0_{k,\alpha,+}(x,x)\neq0,\ \ \forall x\in U, \ \ \mbox{if $q=n_+$},
\end{split}\end{equation}
where $a^0_{k,\alpha, -}(x,x)$ and $a^0_{k,\alpha, +}(x,x)$, $x\in\mu^{-1}(0)\bigcap U$, are the leading terms of the asymptotic expansion of $a_{k, \alpha, -}(x, x)$ and $a_{k, \alpha, +}(x, x)$ in $S^{n-\frac{d}{2}}_{1, 0}(U\times U\times\mathbb{R}_+,T^{*0,q}X\boxtimes(T^{*0,q}X)^*)$, respectively, 
and $\Phi_{k,-}(x,y)\in C^\infty(U\times U)$,
\begin{equation}\label{e-gue170125}
\begin{split}
& {\rm Im\,}\Phi_{k,-}(x,y)\geq0,\\
&d_x\Phi_{k,-}(x,x)=-d_y\Phi_{k,-}(x,x)=-\omega_0(x),\ \ \forall x\in U\bigcap\mu^{-1}(0),\\
\end{split}
\end{equation}
and $-\bar\Phi_{k,+}(x,y)$ satisfies \eqref{e-gue170125}.
\end{theorem}

We refer the readers to the discussion before \eqref{e-gue160507f} and Definition~\ref{d-gue140221a} for the precise meanings of $A\equiv B$ and the symbol space $S^{n-\frac{d}{2}}_{{\rm cl\,}}$, respectively.

 Assume that $G$ acts freely on $\mu^{-1}(0)$ for a moment. To state the formulas for $a^0_{k, -}(x,x)$ and $a^0_{k, +}(x,x)$, we introduce some notations. For a given point $x_0\in X$, let $\{W_j\}_{j=1}^{n}$ be an
orthonormal frame of $(T^{1,0}X,\langle\,\cdot\,|\,\cdot\,\rangle)$ near $x_0$, for which the Levi form
is diagonal at $x_0$. Put
\[
\mathcal{L}_{x_0}(W_j,\ol W_\ell)=\mu_j(x_0)\delta_{j\ell}\,,\;\; j,\ell=1,\ldots,n\,.
\]
We will denote by
\begin{equation}\label{det140530}
\det\mathcal{L}_{x_0}=\prod_{j=1}^{n}\mu_j(x_0)\,.
\end{equation}
Let $\{T_j\}_{j=1}^{n}$ denote the basis of $T^{*0,1}X$, dual to $\{\ol W_j\}^{n}_{j=1}$. We assume that
$\mu_j(x_0)<0$ if\, $1\leq j\leq n_-$ and $\mu_j(x_0)>0$ if\, $n_-+1\leq j\leq n$. Put
\[
\renewcommand{\arraystretch}{1.2}
\begin{array}{ll}
&\mathcal{N}(x_0,n_-):=\set{cT_1(x_0)\wedge\ldots\wedge T_{n_-}(x_0);\, c\in\Complex},\\
&\mathcal{N}(x_0,n_+):=\set{cT_{n_-+1}(x_0)\wedge\ldots\wedge T_{n}(x_0);\, c\in\Complex}
\end{array}
\]
and let
\begin{equation}\label{tau140530}
\tau_{n_-}=\tau_{x_0,n_-}:T^{*0,q}_{x_0}X\To\mathcal{N}(x_0,n_-)\,,\quad
\tau_{n_+}=\tau_{x_0,n_+}:T^{*0,q}_{x_0}X\To\mathcal{N}(x_0,n_+)\,,
\end{equation}
be the orthogonal projections onto $\mathcal{N}(x_0,n_-)$ and $\mathcal{N}(x_0,n_+)$
with respect to $\langle\,\cdot\,|\,\cdot\,\rangle$, respectively.

Fix $x\in\mu^{-1}(0)$, consider the linear map 
\[
\renewcommand{\arraystretch}{1.2}
\begin{array}{rll}
R_x:\underline{\mathfrak{g}}_x&\To&\underline{\mathfrak{g}}_x,\\
u&\To& R_xu,\ \ \langle\,R_xu\,|\,v\,\rangle=\langle\,d\omega_0(x)\,,\,Ju\wedge v\,\rangle.
\end{array}
\]
Let $\det R_x=\lambda_1(x)\cdots\lambda_d(x)$, where $\lambda_j(x)$, $j=1,2,\ldots,d$, are the eigenvalues of $R_x$. 

Fix $x\in\mu^{-1}(0)$, put $Y_x=\set{g\circ x;\, g\in G}$. $Y_x$ is a $d$-dimensional submanifold of $X$. The $G$-invariant Hermitian metric $\langle\,\cdot\,|\,\cdot\,\rangle$ induces a volume form $dv_{Y_x}$ on $Y_x$. Put 
\[
V_{{\rm eff\,}}(x):=\int_{Y_x}dv_{Y_x}.
\]
Note that the function $V_{{\rm eff\,}}(x)$ was already appeared in Ma-Zhang~\cite[(0,10)]{MZI} as exactly the role in the expansion, cf.~\cite[(0.14)]{MZI}.

\begin{theorem}\label{t-gue170128}
With the notations used above, if $G$ acts freely on $\mu^{-1}(0)$, then for $a^0_{k, -}(x,y)$ and $a^0_{k, +}(x,y)$ in \eqref{e-gue170108wrIIIm}, we have 
\begin{equation}\label{e-gue170128}
\begin{split}
a^0_{k, -}(x,x)=2^{d-1}\frac{d^2_k}{V_{{\rm eff\,}}(x)}\pi^{-n-1+\frac{d}{2}}\abs{\det R_x}^{-\frac{1}{2}}\abs{\det\mathcal{L}_{x}}\tau_{x,n_{-}},\ \ \forall x\in\mu^{-1}(0)\\
a^0_{k, +}(x,x)=2^{d-1}\frac{d^2_k}{V_{{\rm eff\,}}(x)}\pi^{-n-1+\frac{d}{2}}\abs{\det R_x}^{-\frac{1}{2}}\abs{\det\mathcal{L}_{x}}\tau_{x,n_{+}},\ \ \forall x\in\mu^{-1}(0).
\end{split}
\end{equation}
\end{theorem}

Assume that $X$ admits an $S^1$ action $e^{i\theta}$: $S^1\times X\rightarrow X$. Let $T\in C^\infty(X, TX)$ be the global real vector field induced by the $S^1$ action given by
$(Tu)(x)=\frac{\partial}{\partial\theta}\left(u(e^{i\theta}\circ x)\right)|_{\theta=0}$, $u\in C^\infty(X)$. Let the $S^1$ action $e^{i\theta}$ be CR and transversal (see Definition~\ref{d-gue160502}). We assume throughout that
\begin{ass}\label{a-gue170128}
	\begin{equation}\label{e-gue170111ryI}
	\mbox{$T$ is transversal to the space $\underline{\mathfrak{g}}$ at every point $p\in\mu^{-1}(0)$},
	\end{equation}
	\begin{equation}\label{e-gue170111ryII}
	e^{i\theta}\circ g\circ x=g\circ e^{i\theta}\circ x,\  \ \forall x\in X,\ \ \forall\theta\in[0,2\pi[,\ \ \forall g\in G, 
	\end{equation}
	and 
	\begin{equation}\label{e-gue170117t}
	\mbox{$G\times S^1$ acts locally freely near $\mu^{-1}(0)$}. 
	\end{equation}
\end{ass}

Let $u\in\Omega^{0,q}(X)$. Define
\begin{equation}\label{e-gue150508faIIm}
Tu:=\frac{\pr}{\pr\theta}\bigr((e^{i\theta})^*u\bigr)|_{\theta=0}\in\Omega^{0,q}(X).
\end{equation}
For every $m\in\mathbb Z$, let
\begin{equation}\label{e-gue150508dIm}
\begin{split}
&\Omega^{0,q}_m(X):=\set{u\in\Omega^{0,q}(X);\, Tu=imu},\ \ q=0,1,2,\ldots,n,\\
&\Omega^{0,q}_{m}(X)_k=\set{u\in\Omega^{0,q}(X)_k;\, Tu=imu},\ \ q=0,1,2,\ldots,n.
\end{split}
\end{equation}
Denote by $C^\infty_m(X):=\Omega^{0,0}_m(X)$, $C^\infty_{m}(X)_k:=\Omega^{0,0}_{m}(X)_k$. From the CR property of the $S^1$ action and \eqref{e-gue170111ryII}, we have
\[Tg^*\ddbar_b=g^*T\ddbar_b=\ddbar_bg^*T=\ddbar_bTg^*\ \ \mbox{on $\Omega^{0,q}(X)$},\ \ \forall g\in G.\]
Hence,
\begin{equation}\label{e-gue160527m}
\ddbar_b:\Omega^{0,q}_{m}(X)_k\To\Omega^{0,q+1}_{m}(X)_k,\ \ \forall m\in\mathbb Z.
\end{equation}
Assume that the Hermitian metric $\langle\,\cdot\,|\,\cdot\,\rangle$ on $\Complex TX$ is $G\times S^1$ invariant.  Then the $L^2$ inner product $(\,\cdot\,|\,\cdot\,)$ on $\Omega^{0,q}(X)$ 
induced by $\langle\,\cdot\,|\,\cdot\,\rangle$ is $G\times S^1$-invariant. We then have 
\[\begin{split}
&Tg^*\ol{\pr}^*_b=g^*T\ol{\pr}^*_b=\ol{\pr}^*_bg^*T=\ol{\pr}^*_bTg^*\ \ \mbox{on $\Omega^{0,q}(X)$},\ \ \forall g\in G,\\
&Tg^*\Box^{(q)}_b=g^*T\Box^{(q)}_b=\Box^{(q)}_bg^*T=\Box^{(q)}_bTg^*\ \ \mbox{on $\Omega^{0,q}(X)$},\ \ \forall g\in G,
\end{split}\]
where $\ol{\pr}^*_b$ is the $L^2$ adjoint of $\ddbar_b$ with respect to $(\,\cdot\,|\,\cdot\,)$. 

Let $L^2_{(0,q), m}(X)_k$ be
the completion of $\Omega_{m}^{0,q}(X)_k$ with respect to $(\,\cdot\,|\,\cdot\,)$. 
Write $L^2_{m}(X)_k:=L^2_{(0,0),m}(X)_k$. Put 
\[H^q_{b,m}(X)_k:=( {\rm Ker\,}\Box^{(q)}_{b} )\bigcap L^2_{(0,q),m}(X)_k.\]
Since $\Box^{(q)}_{b}-T^2$ is elliptic, we have for every $m\in\mathbb Z$, $H^q_{b,m}(X)_k\subset\Omega^{0,q}_{m}(X)_k$ and ${\rm dim\,}H^q_{b,m}(X)_k<\infty$.

\begin{definition}
The $m$-th $G$-equivariant Szeg\H{o} projection is the orthogonal projection 
\[S^{(q)}_{k,m}:L^2_{(0,q)}(X)\To  ( {\rm Ker\,}\Box^{(q)}_{b} )\bigcap L^2_{(0,q),m}(X)_k\]
with respect to $(\,\cdot\,|\,\cdot\,)$. Let $S^{(q)}_{k, m}(x,y)\in C^\infty(X\times X,T^{*0,q}X\boxtimes(T^{*0,q}X)^*)$ be the distribution kernel of $S^{(q)}_{k,m}$. 
\end{definition}

The second main result of this work is the following 

\begin{theorem}\label{t-gue170128I}
	With the assumptions and notations used above, if $q\notin n_-$, then, as $m\To+\infty$, $S^{(q)}_{k,m}=O(m^{-\infty})$ on $X$. 
	
	Suppose  $q=n_-$. Let $D$ be an open neighborhood of $X$ with $D\bigcap\mu^{-1}(0)=\emptyset$. Then, as $m\To+\infty$, 
	\[S^{(q)}_{k,m}=O(m^{-\infty})\ \ \mbox{on $D$}.\]
	Let $p\in\mu^{-1}(0)$ and let $N_p=\{g\in G: g\circ p=p\}=\{g_1=e_0,g_2...,g_r\}$.
	 Let $U$ be an open neighborhood of  $p$ and let $x=(x_1,\ldots,x_{2n+1})$ be local coordinates defined in $U$. 
	Then, as $m\To+\infty$, 
	\begin{equation}\label{e-gue170117pVIIIm}
	\begin{split}
	&S^{(q)}_{k,m}(x,y)\equiv\sum_{\alpha=1}^{r}e^{im\Psi_k(g_\alpha\circ x,y)}b_{k,\alpha}(x,y,m),\\
	&b_{k,\alpha}(x,y,m)\in S^{n-\frac{d}{2}}_{{\rm loc\,}}(1; U\times U, T^{*0,q}X\boxtimes(T^{*0,q}X)^*),\\
	&\mbox{$b_{k,\alpha}(x,y,m)\sim\sum^\infty_{j=0}m^{n-\frac{d}{2}-j}b_{k,\alpha}^j(x,y)$ in $S^{n-\frac{d}{2}}_{{\rm loc\,}}(1; U\times U, T^{*0,q}X\boxtimes(T^{*0,q}X)^*)$},\\
	&b_{k,\alpha}^j(x,y)\in C^\infty(U\times U, T^{*0,q}X\boxtimes(T^{*0,q}X)^*),\ \ j=0,1,2,\ldots,
	\end{split}
	\end{equation}
	$\Psi_k(x,y)\in C^\infty(U\times U)$, $d_x\Psi_k(x,x)=-d_y\Psi_k(x,x)=-\omega_0(x)$, for every $x\in\mu^{-1}(0)$, $\Psi_k(x,y)=0$ if and only if $x=y\in\mu^{-1}(0)$. 
	
	In particular, if $G \times S^1$ acts freely near $\mu^{-1}(0)$, then
	\begin{equation}\label{e-gue170117pVIIIam}
	b_{k}^0(x,x)=2^{d-1}\frac{d^2_k}{V_{{\rm eff\,}}(x)}\pi^{-n-1+\frac{d}{2}}\abs{\det R_x}^{-\frac{1}{2}}\abs{\det\mathcal{L}_{x}}\tau_{x,n_-},\ \ \forall x\in\mu^{-1}(0),
	\end{equation}
	where $\tau_{x,n_-}$ is given by \eqref{tau140530}.
\end{theorem} 

We provide a special case when $G=T^d$ on an irregular Sasakian manifold, where $T^d$
denotes the $d$-dimensional torus. Recall that a CR manifold $X$ is irregular Sasakian if it admits a CR transversal $\mathbb{R}$-action, which does not come from any circle action.
The $\mathbb{R}$-action can be interpreted as a CR torus action $T^{d+1}=T^d\times S^1$ (cf. Herrmann-Hsiao-Li \cite{HHL}). 
Fix $(p_1,...,p_d)\in\mathbb{Z}^d$, we define a $T^d$-action as follows:
\begin{equation}\label{e-1101}
\begin{split}
T^d\times X&\rightarrow X \\
\bigl((e^{i\theta_1},...,e^{i\theta_d} ),x \bigr)&\mapsto
(e^{i\theta_1},...,e^{i\theta_d},e^{-ip_1\theta_1-\cdot\cdot\cdot-ip_d\theta_d})\circ x.
\end{split}
\end{equation}
The $T^d$-action satisfies Assumption \ref{a-gue170128}. 
All irreducible unitary representations of $T^d$ are $\{R_{p_1,...,p_d}:(e^{i\theta_1},...,e^{i\theta_d})\mapsto 
e^{ip_1\theta_1+\cdot\cdot\cdot+ip_1\theta_d} \}$. Let $T_j$ be the induced vector fields of the $T^d$-action. That is, $T_ju(x):=\frac{\pr}{\pr\theta_j}|_{\theta_j=0}u\bigl((1,...,e^{i\theta_j},...,1)\circ x \bigr)$, for $j=1,...,d$ and $u\in C^\infty(X)$. Then
\begin{equation}\label{e-1104}
H_{b,mp_1,...,mp_d,m}^q(X)=\{u\in H_b^q(X):Tu=imu, T_ju=imp_ju,j=1,...,d  \}.
\end{equation}
The $m$-th equivariant Szeg\H{o} kernel is the distributional kernel of the orthogonal projection $S_{mp_1,...,mp_d,m}^q: L^2_{(0,q)}\rightarrow H_{b,mp_1,...,mp_d,m}^q(X)$.

Assume that $(-p_1,...,-p_d)\in\mathbb{Z}^d$ is a regular value of the torus invariant CR moment map
\begin{equation}\label{e-1105}
\mu_0: X\rightarrow\mathbb{R}^d, \mu_0(x):=\bigl(\langle\omega_0(x),T_1(x)\rangle,...,\langle\omega_0(x),T_d(x)\rangle \bigr).
\end{equation}
The CR moment map of $T^d$-action defined in \eqref{e-1101} is 
\begin{equation}\label{e-1106}
\mu: X\rightarrow\mathbb{R}^d, \mu(x):=\bigl(\langle\omega_0(x),-p_1T+T_1(x)\rangle,...,\langle\omega_0(x),-p_dT+T_d(x)\rangle \bigr).
\end{equation}
For $x\in\mu_0^{-1}(-p_1,...,-p_d)$, $\omega_0(-p_jT+T_j)=0$. Then $0$ is the regular value of $\mu$.
By Theorem \ref{t-gue170128I}, we deduce the following which covers Shen's result \cite{Shen} when $X$ is strongly pseudoconvex.
\begin{corollary}\label{c-1101}
Fix $(-p_1,...,-p_d)\in\mathbb{Z}^d$, the action of $T^d$ on the irregular Sasakian manifold $X$ is defined in \eqref{e-1101}. With the assumptions and notations used above, 
 If $q\notin n_-$, then, as $m\To+\infty$, $S^{(q)}_{mp_1,...,mp_d,m}=O(m^{-\infty})$ on $X$. 

Suppose  $q=n_-$. Let $D$ be an open neighborhood of $X$ with $D\bigcap\mu^{-1}(0)=\emptyset$. Then, as $m\To+\infty$, 
\[S^{(q)}_{mp_1,...,mp_d,m}=O(m^{-\infty})\ \ \mbox{on $D$}.\]
Let $p\in\mu^{-1}(0)$ and let $N_p=\{g\in G: g\circ p=p\}=\{g_1=e_0,g_2...,g_r\}$.
Let $U$ be an open neighborhood of  $p$ and let $x=(x_1,\ldots,x_{2n+1})$ be local coordinates defined in $U$. 
Then, as $m\To+\infty$, 
\begin{equation}\label{e-1102}
\begin{split}
&S^{(q)}_{mp_1,...,mp_d,m}(x,y)\equiv\sum_{\alpha=1}^{r}e^{im\Psi(g_\alpha\circ x,y)}b_{\alpha}(x,y,m),\\
&b_{\alpha}(x,y,m)\in S^{n-\frac{d}{2}}_{{\rm loc\,}}(1; U\times U, T^{*0,q}X\boxtimes(T^{*0,q}X)^*),\\
&\mbox{$b_{\alpha}(x,y,m)\sim\sum^\infty_{j=0}m^{n-\frac{d}{2}-j}b_{\alpha}^j(x,y)$ in $S^{n-\frac{d}{2}}_{{\rm loc\,}}(1; U\times U, T^{*0,q}X\boxtimes(T^{*0,q}X)^*)$},\\
&b_{\alpha}^j(x,y)\in C^\infty(U\times U, T^{*0,q}X\boxtimes(T^{*0,q}X)^*),\ \ j=0,1,2,\ldots,
\end{split}
\end{equation}
$\Psi(x,y)\in C^\infty(U\times U)$, $d_x\Psi(x,x)=-d_y\Psi(x,x)=-\omega_0(x)$, for every $x\in\mu^{-1}(0)$, $\Psi(x,y)=0$ if and only if $x=y\in\mu^{-1}(0)$. 

In particular, if $T^d \times S^1$ acts freely near $\mu^{-1}(0)$, then
\begin{equation}\label{e-1103}
b^0(x,x)=\frac{2^{d-1}}{V_{{\rm eff\,}}(x)}\pi^{-n-1+\frac{d}{2}}\abs{\det R_x}^{-\frac{1}{2}}\abs{\det\mathcal{L}_{x}}\tau_{x,n_-},\ \ \forall x\in\mu^{-1}(0),
\end{equation}
where $\tau_{x,n_-}$ is given by \eqref{tau140530}.
\end{corollary}


\section{Preliminaries}\label{s:prelim}
\subsection{Standard notations} \label{s-ssna}

Let $M$ be a $C^\infty$ paracompact manifold.
We let $TM$ and $T^*M$ denote the tangent bundle of $M$
and the cotangent bundle of $M$, respectively.
The complexified tangent bundle of $M$ and the complexified cotangent bundle of $M$ will be denoted by $\Complex TM$
and $\Complex T^*M$, respectively. Write $\langle\,\cdot\,,\cdot\,\rangle$ to denote the pointwise
duality between $TM$ and $T^*M$.
We extend $\langle\,\cdot\,,\cdot\,\rangle$ bilinearly to $\Complex TM\times\Complex T^*M$.

Let $F$ be a $C^\infty$ vector bundle over $M$. The fiber of $F$ at $x\in M$ will be denoted by $F_x$.
Let $E$ be a vector bundle over a $C^\infty$ paracompact manifold $M_1$. We write
$F\boxtimes E^*$ to denote the vector bundle over $M\times M_1$ with fiber over $(x, y)\in M\times M_1$
consisting of the linear maps from $E_y$ to $F_x$.  Let $Y\subset M$ be an open set. 
From now on, the spaces of distribution sections of $F$ over $Y$ and
smooth sections of $F$ over $Y$ will be denoted by $D'(Y, F)$ and $C^\infty(Y, F)$, respectively.
Let $E'(Y, F)$ be the subspace of $D'(Y, F)$ whose elements have compact support in $Y$. Put $C^\infty_c(M, F):=C^\infty(M, F) \cap E'(M, F)$.

We recall the Schwartz kernel theorem \cite[Theorems\,5.2.1, 5.2.6]{Hor03}, \cite[Thorem\,B.2.7]{MM}.
Let $F$ and $E$ be $C^\infty$ vector
bundles over paracompact orientable $C^\infty$ manifolds $M$ and $M_1$, respectively, equipped with smooth densities of integration. If
$A: C^\infty_c(M_1,E)\To D'(M,F)$
is continuous, we write $K_A(x, y)$ or $A(x, y)$ to denote the distribution kernel of $A$.
The following two statements are equivalent
\begin{enumerate}
	\item $A$ is continuous: $E'(M_1,E)\To C^\infty(M,F)$,
	\item $K_A\in C^\infty(M\times M_1,F\boxtimes E^*)$.
\end{enumerate}
If $A$ satisfies (1) or (2), we say that $A$ is smoothing on $M \times M_1$. Let
$A,\hat A: C^\infty_0(M_1,E)\To D'(M,F)$ be continuous operators.
We write 
\begin{equation} \label{e-gue160507f}
\mbox{$A\equiv \hat A$ (on $M\times M_1$)} 
\end{equation}
if $A-\hat A$ is a smoothing operator. If $M=M_1$, we simply write ``on $M$". 

Let $H(x,y)\in D'(M\times M_1,F\boxtimes E^*)$. We write $H$ to denote the unique 
continuous operator $C^\infty_c(M_1,E)\To D'(M,F)$ with distribution kernel $H(x,y)$. 
In this work, we identify $H$ with $H(x,y)$.

\subsection{Some standard notations in semi-classical analysis}\label{s-gue170111w}

Let $W_1$ be an open set in $\Real^{N_1}$ and let $W_2$ be an open set in $\Real^{N_2}$. Let $E$ and $F$ be vector bundles over $W_1$ and $W_2$, respectively. 
An $m$-dependent continuous operator
$A_m:C^\infty_c(W_2,F)\To D'(W_1,E)$ is called $m$-negligible on $W_1\times W_2$
if, for $m$ large enough, $A_m$ is smoothing and, for any $K\Subset W_1\times W_2$, any
multi-indices $\alpha$, $\beta$ and any $N\in\mathbb N$, there exists $C_{K,\alpha,\beta,N}>0$
such that
\begin{equation}\label{e-gue13628III}
\abs{\pr^\alpha_x\pr^\beta_yA_m(x, y)}\leq C_{K,\alpha,\beta,N}m^{-N}\:\: \text{on $K$},\ \ \forall m\gg1.
\end{equation}
In that case we write
\[A_m(x,y)=O(m^{-\infty})\:\:\text{on $W_1\times W_2$,}\]
or
\[A_m=O(m^{-\infty})\:\:\text{on $W_1\times W_2$.}\]
If $A_m, B_m:C^\infty_c(W_2, F)\To D'(W_1, E)$ are $m$-dependent continuous operators,
we write $A_m= B_m+O(m^{-\infty})$ on $W_1\times W_2$ or $A_m(x,y)=B_m(x,y)+O(m^{-\infty})$ on $W_1\times W_2$ if $A_m-B_m=O(m^{-\infty})$ on $W_1\times W_2$. 
When $W=W_1=W_2$, we sometime write ``on $W$".

Let $X$ and $M$ be smooth manifolds and let $E$ and $F$ be vector bundles over $X$ and $M$, respectively. Let $A_m, B_m:C^\infty(M,F)\To C^\infty(X,E)$ be $m$-dependent smoothing operators. We write $A_m=B_m+O(m^{-\infty})$ on $X\times M$ if on every local coordinate patch $D$ of $X$ and local coordinate patch $D_1$ of $M$, $A_m=B_m+O(m^{-\infty})$ on $D\times D_1$.
When $X=M$, we sometime write on $X$.

We recall the definition of the semi-classical symbol spaces

\begin{definition} \label{d-gue140826}
	Let $W$ be an open set in $\Real^N$. Let
	\begin{gather*}
		S(1;W):=\Big\{a\in C^\infty(W)\,|\, \forall\alpha\in\mathbb N^N_0:
		\sup_{x\in W}\abs{\pr^\alpha a(x)}<\infty\Big\},\\
		S^0_{{\rm loc\,}}(1;W):=\Big\{(a(\cdot,m))_{m\in\Real}\,|\, \forall\alpha\in\mathbb N^N_0,
		\forall \chi\in C^\infty_0(W)\,:\:\sup_{m\in\Real, m\geq1}\sup_{x\in W}\abs{\pr^\alpha(\chi a(x,m))}<\infty\Big\}\,.
	\end{gather*}
	For $k\in\Real$, let
	\[
	S^k_{{\rm loc}}(1):=S^k_{{\rm loc}}(1;W)=\Big\{(a(\cdot,m))_{m\in\Real}\,|\,(m^{-k}a(\cdot,m))\in S^0_{{\rm loc\,}}(1;W)\Big\}\,.
	\]
	Hence $a(\cdot,m)\in S^k_{{\rm loc}}(1;W)$ if for every $\alpha\in\mathbb N^N_0$ and $\chi\in C^\infty_0(W)$, there
	exists $C_\alpha>0$ independent of $m$, such that $\abs{\pr^\alpha (\chi a(\cdot,m))}\leq C_\alpha m^{k}$ holds on $W$.
	
	Consider a sequence $a_j\in S^{k_j}_{{\rm loc\,}}(1)$, $j\in\N_0$, where $k_j\searrow-\infty$,
	and let $a\in S^{k_0}_{{\rm loc\,}}(1)$. We say
	\[
	a(\cdot,m)\sim
	\sum\limits^\infty_{j=0}a_j(\cdot,m)\:\:\text{in $S^{k_0}_{{\rm loc\,}}(1)$},
	\]
	if, for every
	$\ell\in\N_0$, we have $a-\sum^{\ell}_{j=0}a_j\in S^{k_{\ell+1}}_{{\rm loc\,}}(1)$ .
	For a given sequence $a_j$ as above, we can always find such an asymptotic sum
	$a$, which is unique up to an element in
	$S^{-\infty}_{{\rm loc\,}}(1)=S^{-\infty}_{{\rm loc\,}}(1;W):=\cap _kS^k_{{\rm loc\,}}(1)$.
	
	Similarly, we can define $S^k_{{\rm loc\,}}(1;Y,E)$ in the standard way, where $Y$ is a smooth manifold and $E$ is a vector bundle over $Y$. 
\end{definition}

\subsection{CR manifolds} 

Let $(X, T^{1,0}X)$ be a compact, connected and orientable CR manifold of dimension $2n+1$, $n\geq 1$, where $T^{1,0}X$ is a CR structure of $X$, that is, $T^{1,0}X$ is a subbundle of rank $n$ of the complexified tangent bundle $\mathbb{C}TX$, satisfying $T^{1,0}X\cap T^{0,1}X=\{0\}$, where $T^{0,1}X=\overline{T^{1,0}X}$, and $[\mathcal V,\mathcal V]\subset\mathcal V$, where $\mathcal V=C^\infty(X, T^{1,0}X)$. There is a unique subbundle $HX$ of $TX$ such that $\mathbb{C}HX=T^{1,0}X \oplus T^{0,1}X$, i.e. $HX$ is the real part of $T^{1,0}X \oplus T^{0,1}X$. Let $J:HX\To HX$ be the complex structure map given by $J(u+\ol u)=iu-i\ol u$, for every $u\in T^{1,0}X$. 
By complex linear extension of $J$ to $\mathbb{C}TX$, the $i$-eigenspace of $J$ is $T^{1,0}X \, = \, \left\{ V \in \mathbb{C}HX \, : \, JV \, =  \,  \sqrt{-1}V  \right\}.$ We shall also write $(X, HX, J)$ to denote a compact CR manifold.

We  fix a real non-vanishing $1$ form $\omega_0\in C(X,T^*X)$ so that $\langle\,\omega_0(x)\,,\,u\,\rangle=0$, for every $u\in H_xX$, for every $x\in X$. 
For each $x \in X$, we define a quadratic form on $HX$ by
\begin{equation}\label{E:levi}
\mathcal{L}_x(U,V) =\frac{1}{2}d\omega_0(JU, V), \forall \ U, V \in H_xX.
\end{equation}
We extend $\mathcal{L}$ to $\mathbb{C}HX$ by complex linear extension. Then, for $U, V \in T^{1,0}_xX$,
\begin{equation}
\mathcal{L}_x(U,\overline{V}) = \frac{1}{2}d\omega_0(JU, \overline{V}) = -\frac{1}{2i}d\omega_0(U,\overline{V}).
\end{equation}
The Hermitian quadratic form $\mathcal{L}_x$ on $T^{1,0}_xX$ is called Levi form at $x$. Let $T\in C^\infty(X,TX)$ be the non-vanishing vector field determined by 
\begin{equation}\label{e-gue170111ry}\begin{split}
&\omega_0(T)=-1,\\
&d\omega_0(T,\cdot)\equiv0\ \ \mbox{on $TX$}.
\end{split}\end{equation}
Note that $X$ is a contact manifold with contact form $\omega_0$, contact plane $HX$ and $T$ is the Reeb vector field.

Fix a smooth Hermitian metric $\langle \cdot \mid \cdot \rangle$ on $\mathbb{C}TX$ so that $T^{1,0}X$ is orthogonal to $T^{0,1}X$, $\langle u \mid v \rangle$ is real if $u, v$ are real tangent vectors, $\langle\,T\,|\,T\,\rangle=1$ and $T$ is orthogonal to $T^{1,0}X\oplus T^{0,1}X$. For $u \in \mathbb{C}TX$, we write $|u|^2 := \langle u | u \rangle$. Denote by $T^{*1,0}X$ and $T^{*0,1}X$ the dual bundles $T^{1,0}X$ and $T^{0,1}X$, respectively. They can be identified with subbundles of the complexified cotangent bundle $\mathbb{C}T^*X$. Define the vector bundle of $(0,q)$-forms by $T^{*0,q}X := \wedge^qT^{*0,1}X$. The Hermitian metric $\langle \cdot | \cdot \rangle$ on $\mathbb{C}TX$ induces, by duality, a Hermitian metric on $\mathbb{C}T^*X$ and also on the bundles of $(0,q)$ forms $T^{*0,q}X, q=0, 1, \cdots, n$. We shall also denote all these induced metrics by $\langle \cdot | \cdot \rangle$. Note that we have the pointwise orthogonal decompositions:
\begin{equation}
\begin{array}{c}
\mathbb{C}T^*X = T^{*1,0}X \oplus T^{*0,1}X \oplus \left\{ \lambda \omega_0: \lambda \in \mathbb{C} \right\}, \\
\mathbb{C}TX = T^{1,0}X \oplus T^{0,1}X \oplus \left\{ \lambda T: \lambda \in \mathbb{C} \right\}.
\end{array}
\end{equation}

For $x, y\in X$, let $d(x,y)$ denote the distance between $x$ and $y$ induced by the Hermitian metric $\langle \cdot \mid \cdot \rangle$. Let $A$ be a subset of $X$. For every $x\in X$, let $d(x,A):=\inf\set{d(x,y);\, y\in A}$. 

Let $D$ be an open set of $X$. Let $\Omega^{0,q}(D)$ denote the space of smooth sections of $T^{*0,q}X$ over $D$ and let $\Omega^{0,q}_0(D)$ be the subspace of $\Omega^{0,q}(D)$ whose elements have compact support in $D$.

\subsection{Fourier analysis on compact Lie groups}\label{subs1}
Let $\rho: G\to GL(\mathbb{C}^d)$ be a representation of $G$, where $d$ is the dimension of the representation $\rho$.
Two representations $\rho_1$ and $\rho_2$ are equivalent if they have the same dimension and there is an invertible matrix $A$ such that $\rho_1(g)=A\rho_2(g)A^{-1}$
for all $g\in G$. Let 
\begin{equation*}
	R=\{R_1, R_2,...\}
\end{equation*}
be the collection of all irreducible unitary representations of $G$, where each $R_k$
comes from exactly only one equivalent class. For each $R_k$, let $(R_{k,j,l})_{j,l=1}^{d_{k}}$ be its matrix, where $d_k$ is the dimension of $R_k$.
Let $d\mu(g)$ be the probability Haar measure on $G$. Let $(\cdot \mid\cdot)_G$ be the natural inner product on $C^{\infty}(G)$ induced by $d\mu(g)$. Let $L^2(G)$ be the completion of $C^{\infty}(G)$ with respect to $(\cdot\mid\cdot)_G$. By the orthogonality relations for compact Lie groups and the Peter-Weyl theorem \cite{Ta}, we have
\begin{theorem}\label{t-10301}
	The set $\{\sqrt{d_k}R_{k,j,l}; j,l=1,...,d_k, k=1,2,...\}$ form an orthonormal basis of $L^2(G)$.
\end{theorem}
For a function $f\in C^{\infty}(G)$, the Fourier component of $f$ with respect to $\sqrt{d_k}R_{k,j,l}$ is 
\begin{equation}\label{e-10311}
f_{k,j,l}:=d_k R_{k,j,l}(g)\int_G f(h)\ol{R_{k,j,l}(h)}d\mu(h)\in C^{\infty}(G).
\end{equation}
The smooth version of the Peter-Weyl theorem on compact Lie groups is the following \cite{Ta}
\begin{theorem}\label{t-10312}
	Let $f\in C^{\infty}(G)$. For every $t\in \mathbb{N}$ and every $\varepsilon>0$,
	there exists a $N_0\in\mathbb{N}$ such that for every $N\geq N_0$, we have
	\begin{equation}\label{e-10316}
	\bigl\|f-\sum_{k=1} ^N\sum_{j,l=1}^{d_k}f_{k,j,l} \bigr\|_{C^t(G)}\leq\varepsilon.
	\end{equation}
\end{theorem}
We put 
\begin{equation*}
	\chi_{m}(g):=\text{Tr}R_k (g)=\sum_{j=1}^{d_k}R_{k,j,j}(g).
\end{equation*}
\begin{definition}\label{d-10311}
	The $k$-th Fourier component of $u\in \Omega^{0,q}(X)$ is defined as
	\begin{equation*}
		u_k(x)=d_k\int_G (g^\star u)(x)\overline{\chi_k(g)}d\mu(g)\in  \Omega^{0,q}(X).
	\end{equation*}
\end{definition}
We have the following theorem about Fourier components. For the readers' convenience,
we present the proof, see also \cite{FHH}.
\begin{theorem}\label{t-10311}
	Let $u\in \Omega^{0,q}(X)$. Then
	\begin{equation}\label{e-10312}
	\lim_{N\to\infty}\sum_{k=1}^N u_k(x)=u(x), \forall x\in X,
	\end{equation}
	\begin{equation}\label{e-10313}
	\langle u_k(x)|u_t(x)\rangle=0, \ \ \text{if} \ \ k\neq t, \forall x\in X,
	\end{equation}
	\begin{equation}\label{e-10314}
	\sum_{k=1}^N\|u_k\|^2\leq \|u\|^2, \forall N\in\mathbb{N}.
	\end{equation}
\end{theorem}
\begin{proof}
	We fix $x\in X$ and consider a smooth function $f: g\in G\to (g^{\star}u)(x)$.
	Then
	\begin{equation}\label{e-10315}
	f_{k,j,l}(g)=d_k R_{k,j,l}(g)\int_G (h^{\star}u)(x)\ol{R_{k,j,l}(h)}d\mu(h).
	\end{equation}
	By Theorem \ref{t-10312}, for every $\varepsilon>0$, there exists a $N_0\in \mathbb{N}$ such that for every $N\geq N_0$, we have
	\begin{equation}\label{e-10317}
	\bigl |(g^{\star}u)(x)-\sum_{k=1} ^N\sum_{j,l=1}^{d_k}f_{k,j,l}(g) \bigr |\leq\varepsilon, \forall g\in G.
	\end{equation}
	Take $g=e_0$, where $e_0$ is the identity element of $g$, we obtain that for every $N\geq N_0$,
	\begin{equation}\label{e-10318}
	\bigl| u(x)-\sum_{k=1} ^N\sum_{j,l=1}^{d_k}f_{k,j,l}(e_0) \bigr|\leq \varepsilon.
	\end{equation}
	Note that by \eqref{e-10315},
	\begin{equation*}
		f_{k,j,l}(e_0)=d_k \delta_{j,l}\int_G (h^{\star}u)(x)\ol{R_{k,j,l}(h)}d\mu(h).
	\end{equation*}
	Then
	\begin{equation}\label{e-10319}
	\sum_{k=1}^N\sum_{j,l=1}^{d_k}f_{k,j,l}(e_0)=\sum_{k=1}^N u_k(x).
	\end{equation}
	Hence \eqref{e-10312} is true by \eqref{e-10318} and \eqref{e-10319}.
	
	By Theorem \ref{t-10301} and \eqref{e-10315}, we have
	\begin{equation}\label{e-103110}
	\begin{split}
	& \sum_{k=1}^{\infty}\sum_{j,l=1}^{d_k}\int_G |f_{k,j,l}(g)|^2d\mu(g) \\
	&=\sum_{k=1}^{\infty}\sum_{j,l=1}^{d_k}d_k\bigl| \int_G (h^{\star}u)(x)\ol{R_{k,j,l}(h)}d\mu(h)\bigr|^2\\
	&=\int_G |(h^{\star}u)(x)|^2 d\mu(h), \forall x\in X.
	\end{split}
	\end{equation}
	Since the metric on $X$ is $G$-invariant, we have
	\begin{equation*}
		\langle p| q \rangle=\langle h^{\star}p|h^{\star}q \rangle, \forall p,q\in \Omega^{0,q}(X), \forall h\in G.
	\end{equation*}
	Then for every $k,t\in\mathbb{N}$,
	\begin{equation}\label{e-103111}
	\langle u_k| u_t \rangle=\int_G \langle h^{\star}u_k| h^{\star}u_t\rangle d\mu(h).
	\end{equation}
	For every $h\in G$, 
	\begin{equation}\label{e-111}
	\begin{split}
	h^{\star}u_k&=d_k \int_G (h^{\star}g^{\star}u)(x)\ol{\chi_{k}(g)}d\mu(g)\\
	&=d_k \int_G ((g\circ h)^{\star}u)(x)\ol{\chi_{k}(g)}d\mu(g)\\
	&=d_k \int_G (g^{\star}u)(x)\ol{\chi_{k}(g\circ h^{-1})}d\mu(g).
	\end{split}
	\end{equation}
	It is easy to see that
	\begin{equation}\label{112}
	\ol{\chi_{k}(g\circ h^{-1})}=\sum_{j=1}^{d_k}\sum_{l=1}^{d_k}\ol{R_{k,j,l}(g)}R_{k,j,l}(h).
	\end{equation}
	Hence
	\begin{equation}\label{113}
	h^{\star}u_k=d_k \int_G (g^{\star}u)(x)(\sum_{j=1}^{d_k}\sum_{l=1}^{d_k}\ol{R_{k,j,l}(g)}R_{k,j,l}(h))d\mu(g).
	\end{equation}
	Similarly,
	\begin{equation}\label{114}
	h^{\star}u_t=d_t \int_G (g^{\star}u)(x)(\sum_{j=1}^{d_t}\sum_{l=1}^{d_t}\ol{R_{t,j,l}(g)}R_{t,j,l}(h))d\mu(g).
	\end{equation}
	From Theorem \ref{t-10301}, \eqref{113} and \eqref{114}, we have
	\begin{equation}\label{115}
	\int_G \langle h^{\star}u_k(x)| h^{\star}u_t(x)\rangle d\mu(h)=0, 
	\forall k\neq t, \forall x\in X.
	\end{equation}
	Then we deduce \eqref{e-10313} from \eqref{e-103111} and \eqref{115}.
	
	For $k=t$, we have
	\begin{equation}\label{116}
	\int_G \langle h^{\star}u_k(x)| h^{\star}u_k(x)\rangle d\mu(h)=
	\sum_{j,l=1}^{d_k}\bigl|\int_G (g^{\star}u)(x)\ol{R_{k,j,l}(g)}d\mu(g) \bigr|^2.
	\end{equation}
	With \eqref{e-103110}, we deduce that for every $N\in\mathbb{N}$ and every $x\in X$,
	\begin{equation}\label{117}
	\sum_{k=1}^{N}\int_G \langle h^{\star}u_k(x)| h^{\star}u_k(x)\rangle d\mu(h)
	\leq \int_G |(g^{\star}u)(x)|^2d\mu(g).
	\end{equation}
	Then for every $N\in\mathbb{N}$,
	\begin{equation*}
		\begin{split}
			\sum_{k=1}^N \|u_k\|^2&=\int_X \sum_{k=1}^N\bigl(\int_G \langle h^{\star}u_k(x)| h^{\star}u_k(x)\rangle d\mu(h)   \bigr)dv_X(x)\\
			&\leq\int_X \int_G |(g^{\star}u)(x)|^2d\mu(g)dv_X(x)=\|u\|^2.
		\end{split}
	\end{equation*}
	The proof is completed.
\end{proof}

We can also prove the following, see \cite[Theorem 3.5]{FHH}.
\begin{theorem}\label{t-111}
	With the notations as above,
	\begin{equation}\label{118}
	\lim_{N\to\infty}\sum_{k=1}^N u_k(x)=u(x)
	\end{equation}
	in $C^{\infty}$-topology.
\end{theorem}
Moreover, we have (see \cite{FHH}),
\begin{proposition}\label{p-11}
	Let $u\in \Omega^{0,q}(X)$, then $u\in\Omega_k^{0,q}(X)$ if and only if
	$u=u_k$ on $X$.
\end{proposition}

\section{$G$-equivariant Szeg\H{o} kernel asymptotics}\label{s-gue161109a}

In this section, we establish asymptotic expansions of the $G$-equivariant Szeg\H{o} kernels. We first review some known results for Szeg\H{o} kernels mainly based on
\cite{Hsiao08} and \cite{HM14}.

\subsection{Szeg\H{o} kernel asymptotics}\label{s-gue161109I}
Fix a Hermitian metric $\langle\,\cdot\,|\,\cdot\,\rangle$ on $\Complex TX$ which induces a Hermitian metric on the bundles of $(0,q)$ 
forms $T^{*0,q}X$, $q=0,1,\ldots,n$.  
Let $D\subset X$ be an open set. Let $\Omega^{0,q}(D)$ denote the space of smooth sections 
of $T^{*0,q}X$ over $D$. 

Let 
\begin{equation} \label{e-suIV}
\ddbar_b:\Omega^{0,q}(X)\To\Omega^{0,q+1}(X)
\end{equation}
be the tangential Cauchy-Riemann operator. 
The natural global $L^2$ inner product $(\,\cdot\,|\,\cdot\,)$ on $\Omega^{0,q}(X)$ 
induced by $dv(x)$ and $\langle\,\cdot\,|\,\cdot\,\rangle$ is given by
\begin{equation}\label{e:l2}
(\,u\,|\,v\,):=\int_X\langle\,u(x)\,|\,v(x)\,\rangle\, dv(x)\,,\quad u,v\in\Omega^{0,q}(X)\,.
\end{equation}
We denote by $L^2_{(0,q)}(X)$ 
the completion of $\Omega^{0,q}(X)$ with respect to $(\,\cdot\,|\,\cdot\,)$. 
Write $L^2(X):=L^2_{(0,0)}(X)$. 
We extend
$\ddbar_{b}$ to $L^2_{(0,r)}(X)$, $r=0,1,\ldots,n$, by
\begin{equation}\label{e-suVII}
\ddbar_{b}:{\rm Dom\,}\ddbar_{b}\subset L^2_{(0,r)}(X)\To L^2_{(0,r+1)}(X)\,,
\end{equation}
where ${\rm Dom\,}\ddbar_{b}:=\{u\in L^2_{(0,r)}(X);\, \ddbar_{b}u\in L^2_{(0,r+1)}(X)\}$ and, for any $u\in L^2_{(0,r)}(X)$, $\ddbar_{b} u$ is defined in the sense of distributions.
We also write
\begin{equation}\label{e-suVIII}
\ol{\pr}^{*}_{b}:{\rm Dom\,}\ol{\pr}^{*}_{b}\subset L^2_{(0,r+1)}(X)\To L^2_{(0,r)}(X)
\end{equation}
to denote the Hilbert adjoint of $\ddbar_{b}$ in the $L^2$ space with respect to $(\,\cdot\,|\,\cdot\, )$.
Let $\Box^{(q)}_{b}$ denote the (Gaffney extension) of the Kohn Laplacian given by
\begin{equation}\label{e-suIX}
\begin{split}
{\rm Dom\,}\Box^{(q)}_{b}=\Big\{s\in L^2_{(0,q)}(X);&\, 
s\in{\rm Dom\,}\ddbar_{b}\cap{\rm Dom\,}\ol{\pr}^{*}_{b},\,
\ddbar_{b}s\in{\rm Dom\,}\ol{\pr}^{*}_{b},
\ol{\pr}^{*}_{b}s\in{\rm Dom\,}\ddbar_{b}\Big\}\,,\\
\Box^{(q)}_{b}s&=\ddbar_{b}\ol{\pr}^{*}_{b}s+\ol{\pr}^{*}_{b}\ddbar_{b}s
\:\:\text{for $s\in {\rm Dom\,}\Box^{(q)}_{b}$}\,.
\end{split}
\end{equation}
By a result of Gaffney, for every $q=0,1,\ldots,n$, $\Box^{(q)}_{b}$ is a positive self-adjoint operator 
(see \cite[Proposition\,3.1.2]{MM}). That is, $\Box^{(q)}_{b}$ is self-adjoint and 
the spectrum of $\Box^{(q)}_{b}$ is contained in $\ol\Real_+$, $q=0,1,\ldots,n$. Let
\begin{equation}\label{e-suXI-I}
S^{(q)}:L^2_{(0,q)}(X)\To{\rm Ker\,}\Box^{(q)}_b
\end{equation}
be the orthogonal projections with respect to the $L^2$ inner product $(\,\cdot\,|\,\cdot\,)$ and let
\begin{equation}\label{e-suXI-II}
S^{(q)}(x,y)\in D'(X\times X,T^{*0,q}X\boxtimes(T^{*0,q}X)^*)
\end{equation}
denote the distribution kernel of $S^{(q)}$. 

We recall H\"ormander symbol spaces. Let $D\subset X$ be a local coordinate patch with local coordinates $x=(x_1,\ldots,x_{2n+1})$. 

\begin{definition}\label{d-gue140221a}
	For $m\in\Real$, $S^m_{1,0}(D\times D\times\mathbb{R}_+,T^{*0,q}X\boxtimes(T^{*0,q}X)^*)$ 
	is the space of all $a(x,y,t)\in C^\infty(D\times D\times\mathbb{R}_+,T^{*0,q}X\boxtimes(T^{*0,q}X)^*)$ 
	such that, for all compact $K\Subset D\times D$ and all $\alpha, \beta\in\mathbb N^{2n+1}_0$, $\gamma\in\mathbb N_0$, 
	there is a constant $C_{\alpha,\beta,\gamma}>0$ such that 
	\[\abs{\pr^\alpha_x\pr^\beta_y\pr^\gamma_t a(x,y,t)}\leq C_{\alpha,\beta,\gamma}(1+\abs{t})^{m-\gamma},\ \ 
	\forall (x,y,t)\in K\times\Real_+,\ \ t\geq1.\]
	Put 
	\[\begin{split}
	&S^{-\infty}(D\times D\times\mathbb{R}_+,T^{*0,q}X\boxtimes(T^{*0,q}X)^*)\\
	&:=\bigcap_{m\in\Real}S^m_{1,0}(D\times D\times\mathbb{R}_+,T^{*0,q}X\boxtimes(T^{*0,q}X)^*).\end{split}\]
	Let $a_j\in S^{m_j}_{1,0}(D\times D\times\mathbb{R}_+,T^{*0,q}X\boxtimes(T^{*0,q}X)^*)$, 
	$j=0,1,2,\ldots$ with $m_j\To-\infty$, as $j\To\infty$. 
	Then there exists $a\in S^{m_0}_{1,0}(D\times D\times\mathbb{R}_+,T^{*0,q}X\boxtimes(T^{*0,q}X)^*)$ 
	unique modulo $S^{-\infty}$, such that 
	$a-\sum^{k-1}_{j=0}a_j\in S^{m_k}_{1,0}(D\times D\times\mathbb{R}_+,T^{*0,q}X\boxtimes(T^{*0,q}X)^*\big)$ 
	for $k=0,1,2,\ldots$. 
	
	If $a$ and $a_j$ have the properties above, we write $a\sim\sum^{\infty}_{j=0}a_j$ in 
	$S^{m_0}_{1,0}\big(D\times D\times\mathbb{R}_+,T^{*0,q}X\boxtimes(T^{*0,q}X)^*\big)$. 
	We write
	\begin{equation}  \label{e-gue140205III}
	s(x, y, t)\in S^{m}_{{\rm cl\,}}\big(D\times D\times\mathbb{R}_+,T^{*0,q}X\boxtimes(T^{*0,q}X)^*\big)
	\end{equation}
	if $s(x, y, t)\in S^{m}_{1,0}\big(D\times D\times\mathbb{R}_+,T^{*0,q}X\boxtimes(T^{*0,q}X)^*\big)$ and 
	\begin{equation}\label{e-fal}\begin{split}
	&s(x, y, t)\sim\sum^\infty_{j=0}s^j(x, y)t^{m-j}\text{ in }S^{m}_{1, 0}
	\big(D\times D\times\mathbb{R}_+\,,T^{*0,q}X\boxtimes(T^{*0,q}X)^*\big)\,,\\
	&s^j(x, y)\in C^\infty\big(D\times D,T^{*0,q}X\boxtimes(T^{*0,q}X)^*\big),\ j\in\N_0.\end{split}\end{equation}
\end{definition}

The following was proved in Theorem 4.8 in~\cite{HM14}

\begin{theorem}\label{t-gue161109}
	Given $q=0,1,2,\ldots,n$. Assume that 
	$q\notin\set{n_-,n_+}$. Then, $S^{(q)}\equiv0$ on $X$. 
\end{theorem}

We have the following (see Theorem 1.2 in~\cite{Hsiao08}, Theorem 4.7 in~\cite{HM14} and see also~\cite{BouSj76} for $q=0$)

\begin{theorem}\label{t-gue161109I}
	Let $q=n_-$ or $n_+$. Suppose that $\Box^{(q)}_b$ has $L^2$ closed range. Then, 
	$S^{(q)}(x,y)\in C^\infty(X\times X\setminus{{\rm diag\,}(X\times X)},T^{*0,q}X\boxtimes(T^{*0,q}X)^*)$.
	Let $D\subset X$ be any local coordinate patch with local coordinates $x=(x_1,\ldots,x_{2n+1})$. Then, there exist continuous operators
	$S_-, S_+:\Omega^{0,q}_0(D)\To D'(D,T^{*0,q}X)$
	such that 
	\begin{equation}\label{e-gue161110}
	S^{(q)}\equiv S_-+S_+\ \ \mbox{on $D$},
	\end{equation}
	and $S_-(x,y)$, $S_+(x,y)$ satisfy
	\[\begin{split}
	&S_-(x, y)\equiv\int^{\infty}_{0}e^{i\varphi_-(x, y)t}s_-(x, y, t)dt\ \ \mbox{on $D$},\\
	&S_+(x, y)\equiv\int^{\infty}_{0}e^{i\varphi_+(x, y)t}s_+(x, y, t)dt\ \ \mbox{on $D$},
	\end{split}\]
	with 
	\begin{equation}  \label{e-gue161110r}\begin{split}
	&s_-(x, y, t), s_+(x,y,t)\in S^{n}_{1,0}(D\times D\times\mathbb{R}_+,T^{*0,q}X\boxtimes(T^{*0,q}X)^*), \\
	&s_-(x,y,t)=0\ \ \mbox{if $q\neq n_-$},\ \ s_+(x,y,t)=0\ \ \mbox{if $q\neq n_+$},\\
	&s_-(x, y, t)\sim\sum^\infty_{j=0}s^j_-(x, y)t^{n-j}\quad\text{ in }S^{n}_{1, 0}(D\times D\times\mathbb{R}_+,T^{*0,q}X\boxtimes(T^{*0,q}X)^*),\\
	&s_+(x, y, t)\sim\sum^\infty_{j=0}s^j_+(x, y)t^{n-j}\quad\text{ in }S^{n}_{1, 0}(D\times D\times\mathbb{R}_+,T^{*0,q}X\boxtimes(T^{*0,q}X)^*),\\
	&s^j_+(x,y), s^j_-(x, y)\in C^\infty(D\times D,T^{*0,q}X\boxtimes(T^{*0,q}X)^*),\ \ j=0,1,2,3,\ldots,\\
	&s^0_-(x,x)\neq0,\ \ \forall x\in D,\ \ s^0_+(x,x)\neq0,\ \ \forall x\in D,
	\end{split}\end{equation}
	and the phase functions $\varphi_-$, $\varphi_+$  satisfy
	\begin{equation}\label{e-gue140205IV}
	\begin{split}
	&\varphi_+(x,y), \varphi_-\in C^\infty(D\times D),\ \ {\rm Im\,}\varphi_-(x, y)\geq0,\\
	&\varphi_-(x, x)=0,\ \ \varphi_-(x, y)\neq0\ \ \mbox{if}\ \ x\neq y,\\
	&d_x\varphi_-(x, y)\big|_{x=y}=-\omega_0(x), \ \ d_y\varphi_-(x, y)\big|_{x=y}=\omega_0(x), \\
	&\varphi_-(x, y)=-\ol\varphi_-(y, x), \\
	&-\ol\varphi_+(x, y)=\varphi_-(x,y).
	\end{split}
	\end{equation}
\end{theorem}

\begin{remark}\label{r-gue180404}
	Note that for a strictly pseudoconvec CR manifold of dimension $3$, $\Box^{(0)}_b$ does not have $L^2$ closed range in general (see~\cite{Ros:65}). Kohn~\cite{Koh86} proved that if $q=n_-=n_+$ or $\abs{n_--n_+}>1$ then $\Box^{(q)}_b$ has $L^2$ closed range.
\end{remark}

\subsection{$G$-equivariant Szeg\H{o} kernel}\label{s-gue161109}
Since $G$ preserves $J$ and $(\,\cdot\,|\,\cdot\,)$ is $G$-invariant, it is straightforward to see that for all $g\in G$
\begin{equation}\label{e-gue161231}
\begin{split}
&g^*\ddbar_b=\ddbar_bg^*\ \ \mbox{on $\Omega^{0,q}(X)$}, \\
&g^*\ol{\pr}^*_b=\ol{\pr}^*_bg^*\ \ \mbox{on $\Omega^{0,q}(X)$}, \\
&g^*\Box^{(q)}_b=\Box^{(q)}_bg^*\ \ \mbox{on $\Omega^{0,q}(X)$}.
\end{split}
\end{equation}
Denote by $\ddbar_{b,k}$ (resp. $\Box_{b,k}^{(q)}$) the restriction of
$\ddbar_{b}$ (resp. $\Box_{b}^{(q)}$) on $\Omega_k^{0,q}(X)$.
By Definition \ref{d-10311}, we have
\begin{equation}\label{e-gue161231II}
\begin{split}
&\ddbar_{b,k}: \Omega_k^{0,q}(X)\to \Omega_k^{0,q+1}(X),\\
&\Box_{b,k}^{(q)}: \Omega_k^{0,q}(X)\to \Omega_k^{0,q}(X).
\end{split}
\end{equation}

The $G$-equivariant Szeg\H{o} projection is the orthogonal projection 
\[
S^{(q)}_k:L^2_{(0,q)}(X)\To {\rm Ker\,}\Box^{(q)}_b\bigcap L^2_{(0,q),k}(X)
\]
with respect to $(\,\cdot\,|\,\cdot\,)$. Let $S^{(q)}_k(x,y)\in D'(X\times X,T^{*0,q}X\boxtimes(T^{*0,q}X)^*)$ be its distribution kernel. 
\begin{lemma}\label{l-112}
	\begin{equation}\label{e-112}
	S^{(q)}_k(x,y)=d_k\int_G S^{(q)}(g\circ x, y)\ol{\chi_k(g)}d\mu(g).
	\end{equation}
\end{lemma}
\begin{proof}
	Let
	\begin{equation}\label{e-1121}
	\begin{split}
	Q_k:&L^2_{(0,q)}(X)\to L^2_{(0,q),k}(X)\\
	&u\to u_k=d_k\int_G (g^{\star}u)(x)\ol{\chi_k(g)}d\mu(g).
	\end{split}
	\end{equation}
	Then $S^{(q)}_k=Q_k\circ S^{(q)}$.
	For $u\in L^2_{(0,q)}(X)$, we have
	\begin{equation}\label{e-1122}
	\begin{split}
	S^{(q)}_k u&=Q_k\circ S^{(q)}u\\
	&=Q_k\int S^{(q)}(x,y)u(y)dy\\
	&=d_k\int_G g^{\star}(\int S^{(q)}(x,y)u(y)dy )\ol{\chi_k(g)}d\mu(g)\\
	&=d_k\int_G (\int S^{(q)}(g\circ x,y)u(y)dy )\ol{\chi_k(g)}d\mu(g)\\
	&=\int (d_k\int_G S^{(q)}(g\circ x, y)\ol{\chi_k(g)}d\mu(g)) u(y)dy.
	\end{split}
	\end{equation}
	Then the proof is completed.
\end{proof}
Note that
\begin{equation}\label{e-1172}
\begin{split}
S^{(q)}_k(h\circ x,y)&=d_k\int_G S^{(q)}(g\circ h\circ x, y)\ol{\chi_k(g)}d\mu(g)\\
&=d_k\int_G S^{(q)}(g\circ x, y)\ol{\chi_k(g\circ h^{-1})}d\mu(g)\\
&=d_k\int_G S^{(q)}(g\circ x, y)\sum_{j=1}^{d_k}\sum_{l=1}^{d_k}
\ol{R_{k,j,l}(g)}R_{k,j,l}(h)d\mu(g).
\end{split}
\end{equation}
So $S^{(q)}_k(x,y)$ is not $G$-invariant.

\subsection{$G$-equivariant Szeg\H{o} kernels near $\mu^{-1}(0)$}\label{s-gue161110w}

In this subsection, we will study $G$-equivariant Szeg\H{o} kernel near $\mu^{-1}(0)$. 
Let $e_0\in G$ be the identity element. 
Let $v=(v_1,\ldots,v_d)$ be the local coordinates of $G$ defined  in a neighborhood $V$ of $e_0$ with $v(e_0)=(0,\ldots,0)$. From now on, we will identify the element $e\in V$ with $v(e)$. We recall the following on group actions in local coordinates, see \cite[Theorem 3.6]{HH}.

\begin{theorem}\label{t-gue161202}
	Let $p\in\mu^{-1}(0)$. There exist local coordinates $v=(v_1,\ldots,v_d)$ of $G$ defined in  a neighborhood $V$ of $e_0$ with $v(e_0)=(0,\ldots,0)$, local coordinates $x=(x_1,\ldots,x_{2n+1})$ of $X$ defined in a neighborhood $U=U_1\times U_2$ of $p$ with $0\leftrightarrow p$, where $U_1\subset\Real^d$ is an open set of $0\in\Real^d$,  $U_2\subset\Real^{2n+1-d}$ is an open set of $0\in\Real^{2n+1-d} $ and a smooth function $\gamma=(\gamma_1,\ldots,\gamma_d)\in C^\infty(U_2,U_1)$ with $\gamma(0)=0\in\Real^d$  such that
	\begin{equation}\label{e-gue161202}
	\begin{split}
	&(v_1,\ldots,v_d)\circ (\gamma(x_{d+1},\ldots,x_{2n+1}),x_{d+1},\ldots,x_{2n+1})\\
	&=(v_1+\gamma_1(x_{d+1},\ldots,x_{2n+1}),\ldots,v_d+\gamma_d(x_{d+1},\ldots,x_{2n+1}),x_{d+1},\ldots,x_{2n+1}),\\
	&\forall (v_1,\ldots,v_d)\in V,\ \ \forall (x_{d+1},\ldots,x_{2n+1})\in U_2,
	\end{split}
	\end{equation}
	\begin{equation}\label{e-gue161206}
	\begin{split}
	&\underline{\mathfrak{g}}={\rm span\,}\set{\frac{\pr}{\pr x_1},\ldots,\frac{\pr}{\pr x_d}},\\
	&\mu^{-1}(0)\bigcap U=\set{x_{d+1}=\cdots=x_{2d}=0},\\
	&\mbox{On $\mu^{-1}(0)\bigcap U$, we have $J(\frac{\pr}{\pr x_j})=\frac{\pr}{\pr x_{d+j}}+a_j(x)\frac{\pr}{\pr x_{2n+1}}$, $j=1,2,\ldots,d$}, 
	\end{split}
	\end{equation}
	where $a_j(x)$ is a smooth function on $\mu^{-1}(0)\bigcap U$, independent of $x_1,\ldots,x_{2d}$, $x_{2n+1}$ and $a_j(0)=0$, $j=1,\ldots,d$, 
	\begin{equation}\label{e-gue161202I}
	\begin{split}
	&T^{1,0}_pX={\rm span\,}\set{Z_1,\ldots,Z_n},\\
	&Z_j=\frac{1}{2}(\frac{\pr}{\pr x_j}-i\frac{\pr}{\pr x_{d+j}})(p),\ \ j=1,\ldots,d,\\
	&Z_j=\frac{1}{2}(\frac{\pr}{\pr x_{2j-1}}-i\frac{\pr}{\pr x_{2j}})(p),\ \ j=d+1,\ldots,n,\\
	&\langle\,Z_j\,|\,Z_l\,\rangle=\delta_{j,l},\ \ j,l=1,2,\ldots,n,\\
	&\mathcal{L}_p(Z_j, \ol Z_l)=\mu_j\delta_{j,l},\ \ j,l=1,2,\ldots,n
	\end{split}
	\end{equation}
	and 
	\begin{equation}\label{e-gue161219}
	\begin{split}
	\omega_0(x)&=(1+O(\abs{x}))dx_{2n+1}+\sum^d_{j=1}4\mu_jx_{d+j}dx_j\\
	&\quad+\sum^n_{j=d+1}2\mu_jx_{2j}dx_{2j-1}-\sum^n_{j=d+1}2\mu_jx_{2j-1}dx_{2j}+\sum^{2n}_{j=d+1}b_jx_{2n+1}dx_j+O(\abs{x}^2),
	\end{split}
	\end{equation}
	where $b_{d+1}\in\Real,\ldots,b_{2n}\in\Real$. 
\end{theorem}

The following describes the phase function in Theorem~\ref{t-gue161109I}, 
see \cite[Theorem 3.7]{HH}.
\begin{theorem}\label{t-gue161222}
	Let $p\in\mu^{-1}(0)$ and take local coordinates $x=(x_1,\ldots,x_{2n+1})$ of $X$ defined in an open set $U$of $p$ with $0\leftrightarrow p$ such that \eqref{e-gue161206}, \eqref{e-gue161202I} and \eqref{e-gue161219} hold.  Let $\varphi_-(x,y)\in C^\infty(U\times U)$ be as in Theorem~\ref{t-gue161109I}. Then, 
	\begin{equation} \label{e-gue161222}
	\begin{split}
	\varphi_-(x, y)&=-x_{2n+1}+y_{2n+1}-2\sum^d_{j=1}\mu_jx_jx_{d+j}+2\sum^d_{j=1}\mu_jy_jy_{d+j}
	+i\sum^{n}_{j=1}\abs{\mu_j}\abs{z_j-w_j}^2 \\
	&+\sum^{n}_{j=1}i\mu_j(\ol z_jw_j-z_j\ol w_j)+\sum^d_{j=1}(-\frac{i}{2}b_{d+j})(-z_jx_{2n+1}+w_jy_{2n+1})\\
	&+\sum^d_{j=1}(\frac{i}{2}b_{d+j})(-\ol z_jx_{2n+1}+\ol w_jy_{2n+1})+\sum^n_{j=d+1}\frac{1}{2}(b_{2j-1}-ib_{2j})(-z_jx_{2n+1}+w_jy_{2n+1})\\
	&+\sum^n_{j=d+1}\frac{1}{2}(b_{2j-1}+ib_{2j})(-\ol z_jx_{2n+1}+\ol w_jy_{2n+1})+(x_{2n+1}-y_{2n+1})f(x, y) +O(\abs{(x, y)}^3),
	\end{split}
	\end{equation}
	where $z_j=x_j+ix_{d+j}$, $j=1,\ldots,d$, $z_j=x_{2j-1}+ix_{2j}$, $j=2d+1,\ldots,2n$, $\mu_j$, $j=1,\ldots,n$, and $b_{d+1}\in\Real,\ldots,b_{2n}\in\Real$ are as in \eqref{e-gue161219} and $f$ is smooth and satisfies $f(0,0)=0$, $f(x, y)=\ol f(y, x)$. 
\end{theorem}

We now study $S^{(q)}_k(x,y)$. From Theorem~\ref{t-gue161109} and Lemma \ref{l-112}, we get 

\begin{theorem}\label{t-gue170112}
	Assume that $q\notin\set{n_-,n_+}$.Then, $S^{(q)}_k\equiv0$ on $X$. 
\end{theorem}

Assume that $q=n_-$ and $\Box^{(q)}_b$ has $L^2$ closed range. Fix $p\in\mu^{-1}(0)$ and let $v=(v_1,\ldots,v_d)$ and $x=(x_1,\ldots,x_{2n+1})$ be the local coordinates of $G$ and $X$ as in Theorem~\ref{t-gue161202}. 
Assume that $d\mu=m(v)dv=m(v_1,\ldots,v_d)dv_1\cdots dv_d$ on $V$, where $V$ is an open neighborhood of $e_0\in G$ as in Theorem~\ref{t-gue161202}. From Lemma \ref{l-112}, we have 
\begin{equation}\label{e-1171}
S^{(q)}_k(x,y)=d_k\int_G\chi(g)\ol{\chi_k(g)}S^{(q)}(g\circ x, y)d\mu(g)+d_k\int_G(1-\chi(g))\ol{\chi_k(g)}S^{(q)}(g\circ x,y)d\mu(g),
\end{equation}
where $\chi\in C^\infty_0(V)$, $\chi=1$ near $e_0$. 

Assume first $G$ is globally free on $\mu^{-1}(0)$, 
if $U$ and $V$ are small, there is a constant $c>0$ such that 
\begin{equation}\label{e-gue161231cr}
d(g\circ x,y)\geq c,\ \ \forall x, y\in U, g\in{\rm Supp\,}(1-\chi),
\end{equation}
where $U$ is an open set of $p\in\mu^{-1}(0)$ as in Theorem~\ref{t-gue161202}. From now on, we take $U$ and $V$ small enough so that \eqref{e-gue161231cr} holds. By Theorem~\ref{t-gue161109I}, $S^{(q)}(x,y)$ is smoothing away from diagonal. From this observation and \eqref{e-gue161231cr}, we have $\int_G(1-\chi(g))\ol{\chi_k(g)}S^{(q)}(g\circ x, y)d\mu(g)\equiv0$ on $U$ 
and hence 
\begin{equation}\label{e-gue170102}
\mbox{$S^{(q)}_k(x,y)\equiv d_k\int_G\chi(g)\ol{\chi_k(g)}S^{(q)}(g\circ x, y)d\mu(g)$ on $U$}.
\end{equation}
From Theorem~\ref{t-gue161109I} and \eqref{e-gue170102}, we have 
\begin{equation}\label{e-gue170102I}
\begin{split}
&\mbox{$S^{(q)}_k(x,y)\equiv \hat S_{k,-}(x,y)+\hat S_{k,+}(x,y)$ on $U$},\\
&\hat S_{k,-}(x,y)=d_k\int_G\chi(g)\ol{\chi_k(g)}S_-(g\circ x,y)d\mu(g),\\
&\hat S_{k,+}(x,y)=d_k\int_G\chi(g)\ol{\chi_k(g)}S_+(g\circ x,y)d\mu(g).
\end{split}
\end{equation}
More precisely,
\begin{equation}\label{e-1127}
\hat S_{k,-}(x,y)\equiv d_k\int_0^{\infty}\int_V
e^{i\phi_-(v\circ x,y)t}\chi(v)\ol{\chi_k(v)}s_-(v\circ x,y,t)m(v)dvdt.
\end{equation}

By using stationary phase formula of Melin-Sj\"ostrand~\cite{MS74}, it follows from the arguments in \cite[Section 3.3]{HH} that
\begin{equation}\label{e-gue170102cw}
\hat S_{k,-}(x,y)\equiv\int e^{i\Phi_{k,-}(x,y)t}a_{k,-}(x,y,t)dt\ \ \ \mbox{on $U$},
\end{equation}
where $a_{k,-}(x,y,t)\sim\sum^\infty_{j=0}t^{n-\frac{d}{2}-j}a^j_{k,-}(x,y)$ in $S^{n-\frac{d}{2}}_{1,0}(U\times U\times\Real_+, T^{*0,q}X\boxtimes(T^{*0,q}X)^*)$, 
\[a^j_{k,-}(x,y)\in C^\infty(U\times U,T^{*0,q}X\boxtimes(T^{*0,q}X)^*),\ \ j=0,1,\ldots,\]

In this work, $G$ acts locally free on $\mu^{-1}(0)$ under Assumption \ref{a-gue170123II}. Let $N_p=\{g\in G: g\circ p=p\}=\{g_1=e_0,g_2...,g_r\}$.
Similarly to \eqref{e-gue161231cr},
we can choose $U$ and $V$ to be small such that 
the subsets $\{g_j V\}_{\alpha=1}^r$ are mutually disjoint and
there is a constant $c>0$ satisfying 
\begin{equation}\label{e-11271}
d(h\circ x,y)\geq c,\ \ \forall x, y\in U, h\in{\rm Supp\,}(1-\sum_{\alpha=1}^r \chi(gg_\alpha^{-1})).
\end{equation}
Then on $U$
\begin{equation}\label{e-11272}
\begin{split}
S^{(q)}_k(x,y)&\equiv d_k\sum_{\alpha=1}^r\int_G\chi(gg_\alpha^{-1})\ol{\chi_k(g)}S^{(q)}(g\circ x, y)d\mu(g)\\
&=d_k\sum_{\alpha=1}^r\int_G\chi(g)\ol{\chi_k(gg_\alpha)}S^{(q)}(gg_\alpha\circ x, y)d\mu(g).
\end{split}
\end{equation}
\begin{equation}\label{e-11273}
\hat S_{k,-}(x,y)\equiv d_k\sum_{\alpha=1}^r\int_0^{\infty}\int_V
e^{i\phi_-(vv_\alpha\circ x,y)t}\chi(v)\ol{\chi_k(vv_\alpha)}s_-(vv_\alpha\circ x,y,t)m(v)dvdt,
\end{equation}
where $v_\alpha$ is the coordinate of $g_\alpha$.
By using stationary phase formula of Melin-Sj\"ostrand~\cite{MS74}, it follows from the above argument that
\begin{equation}\label{e-11274}
\hat S_{k,-}(x,y)\equiv\sum_{\alpha=1}^r\int e^{i\Phi_{k,-}(v_\alpha\circ x,y)t}a_{k,\alpha,-}(x,y,t)dt\ \ \ \mbox{on $U$},
\end{equation}
where $a_{k,\alpha,-}(x,y,t)\sim\sum^\infty_{j=0}t^{n-\frac{d}{2}-j}a^j_{k,\alpha,-}(x,y)$ in $S^{n-\frac{d}{2}}_{1,0}(U\times U\times\Real_+, T^{*0,q}X\boxtimes(T^{*0,q}X)^*)$, 
\[a^j_{k,\alpha,-}(x,y)\in C^\infty(U\times U,T^{*0,q}X\boxtimes(T^{*0,q}X)^*),\ \ j=0,1,\ldots,\]

In particular, it follows from the arguments in \cite[Subsection 3.3]{HH} with minor modification, if $G$ acts freely on $\mu^{-1}(0)$, then for $a^0_{k, -}(x,y)$ and $a^0_{k, +}(x,y)$ in \eqref{e-gue170108wrIIIm}, we have 
\[
\begin{split}
a^0_{k, -}(x,x)=2^{d-1}\frac{d^2_k}{V_{{\rm eff\,}}(x)}\pi^{-n-1+\frac{d}{2}}\abs{\det R_x}^{-\frac{1}{2}}\abs{\det\mathcal{L}_{x}}\tau_{x,n_{-}},\ \ \forall x\in\mu^{-1}(0)\\
a^0_{k, +}(x,x)=2^{d-1}\frac{d^2_k}{V_{{\rm eff\,}}(x)}\pi^{-n-1+\frac{d}{2}}\abs{\det R_x}^{-\frac{1}{2}}\abs{\det\mathcal{L}_{x}}\tau_{x,n_{+}},\ \ \forall x\in\mu^{-1}(0).
\end{split}
\]

\subsection{$G$-equivariant Szeg\H{o} kernel asymptotics away $\mu^{-1}(0)$}\label{s-gue170110}

The goal of this section is to prove the following 

\begin{theorem}\label{t-gue170110}
	Let $D$ be an open neighborhood of $X$ with $D\bigcap\mu^{-1}(0)=\emptyset$. Then, 
	\[S^{(q)}_k\equiv0\ \ \mbox{on $D$}.\]
\end{theorem}

We first need 
\begin{lemma}\label{l-gue170110}
	Let $p\notin\mu^{-1}(0)$. Then, there are open neighborhoods $U$ of $p$ and $V$ of $e\in G$ such that for any $\chi\in C^\infty_0(V)$, we have for every $k$,
	\begin{equation}\label{e-gue170110}
	\int_GS^{(q)}(x,g\circ y)\chi(g)\ol{\chi_k(g)}d\mu(g)\equiv0\ \ \mbox{on $U$}.
	\end{equation}
\end{lemma}
The proof of the above lemma follows from \cite[Lemma 3.14]{HH} by adding $\ol{\chi_k(g)}$.

\begin{lemma}\label{l-gue170111}
	Let $p\notin\mu^{-1}(0)$ and let $h\in G$. We can find open neighborhoods $U$ of $p$ and $V$ of $h$ such that for every $\chi\in C^\infty_0(V)$, we have for every $k$,
	\[\int_GS^{(q)}(g\circ x,y)\chi(g) \ol{\chi_k(g)} d\mu(g)\equiv0\ \ \mbox{on $U$}.\]
\end{lemma}

\begin{proof}
	Let $U$ and $V$ be open sets as in Lemma~\ref{l-gue170110}. Let $\hat V=hV$. Then, $\hat V$ is an open set of $G$. Let $\hat\chi\in C^\infty_0(\hat V)$. We have 
	\begin{equation}\label{e-gue170111}
	\begin{split}
	\int_GS^{(q)}(g\circ x,y)\hat\chi(g)\ol{\chi_k(g)}d\mu(g)&=\int_GS^{(q)}(h\circ g\circ x, y)\hat\chi(h\circ g)\ol{\chi_k(h\circ g)}d\mu(g)\\
	&=\int_GS^{(q)}(h\circ g\circ x,y)\chi(g)d\mu(g),
	\end{split}
	\end{equation}
	where $\chi(g):=\hat\chi(h\circ g)\ol{\chi_k(h\circ g)}\in C^\infty_0(V)$. From \eqref{e-gue170111} and Lemma~\ref{l-gue170110}, we deduce that 
	\[\int_GS^{(q)}(g\circ x,y)\hat\chi(g)\ol{\chi_k(g)}d\mu(g)\equiv0\ \ \mbox{on $U$}.\]
	The lemma follows. 
\end{proof}

\begin{proof}[Proof of Theorem~\ref{t-gue170110}]
	Fix $p\in D$. We need to show that $S^{(q)}_k$ is smoothing near $p$. 
	Let $h\in G$. By Lemma~\ref{l-gue170111}, we can find open sets $U_h$ of $p$ and $V_h$ of $h$ such that for every $\chi\in C^\infty_0(V_h)$, we have 
	\begin{equation}\label{e-gue170111c}
	\int_GS^{(q)}(g\circ x,y)\chi(g)\ol{\chi_k(g)}d\mu(g)\equiv0\ \ \mbox{on $U_h$}.
	\end{equation}
	Since $G$ is compact, we can find open sets $U_{h_j}$ and $V_{h_j}$, $j=1,\ldots,N$, such that $G=\bigcup^N_{j=1}V_{h_j}$. Let $U=D\bigcap\Bigr(\bigcap^N_{j=1}U_{h_j}\Bigr)$ and let $\tilde\chi_j\in C^\infty_0(V_{h_j})$, $j=1,\ldots,N$, with $\sum^N_{j=1}\tilde\chi_j=1$ on $G$. From \eqref{e-gue170111c}, we have 
	\begin{equation}\label{e-gue170111cI}
	\begin{split}
	S^{(q)}_k(x,y)&=d_k\int_G S^{(q)}(g\circ x, y)\ol{\chi_k(g)}d\mu(g)\\
	&=d_k\sum^N_{j=1}\int_GS^{(q)}(g\circ x,y)\tilde\chi_j(g)\ol{\chi_k(g)}d\mu(g)\equiv0\ \ \mbox{on $U$}.
	\end{split}
	\end{equation}
	The theorem follows. 
\end{proof}

From Section \ref{s-gue161110w} and Section \ref{s-gue170110}, we get Theorem~\ref{t-gue170124}. 

\section{$G$-equivariant Szeg\H{o} kernel asymptotics on CR manifolds with $S^1$ action}\label{s-gue170111}

Let $X$ admit an $S^1$ action $e^{i\theta}$: $S^1\times X\rightarrow X$. Let $T\in C^\infty(X, TX)$ be the global real vector field induced by the $S^1$ action given by
$(Tu)(x)=\frac{\partial}{\partial\theta}\left(u(e^{i\theta}\circ x)\right)|_{\theta=0}$, $u\in C^\infty(X)$. 

\begin{definition}\label{d-gue160502}
	The $S^1$ action $e^{i\theta}$ is CR if
	$[T, C^\infty(X, T^{1,0}X)]\subset C^\infty(X, T^{1,0}X)$ and the $S^1$ action is transversal if for each $x\in X$,
	$\Complex T(x)\oplus T_x^{1,0}X\oplus T_x^{0,1}X=\mathbb CT_xX$. Moreover, the $S^1$ action is locally free if $T\neq0$ everywhere. 
\end{definition}

Note that transversality implies local freeness. Let $(X, T^{1,0}X)$ be a compact connected CR manifold with a transversal CR $S^1$ action $e^{i\theta}$ and $T$ be the global vector field induced by the $S^1$ action. Let $\omega_0\in C^\infty(X,T^*X)$ be the global real one form determined by $\langle\,\omega_0\,,\,u\,\rangle=0$, for every $u\in T^{1,0}X\oplus T^{0,1}X$, and $\langle\,\omega_0\,,\,T\,\rangle=-1$. Note that $\omega_0$ and $T$ satisfy \eqref{e-gue170111ry}. Recall that we work with Assumption~\ref{a-gue170128}. 

Assume that the Hermitian metric $\langle\,\cdot\,|\,\cdot\,\rangle$ on $\Complex TX$ is $G\times S^1$ invariant.  Then the $L^2$ inner product $(\,\cdot\,|\,\cdot\,)$ on $\Omega^{0,q}(X)$ 
induced by $\langle\,\cdot\,|\,\cdot\,\rangle$ is $G\times S^1$-invariant. We then have 
\[\begin{split}
&Tg^*\ol{\pr}^*_b=g^*T\ol{\pr}^*_b=\ol{\pr}^*_bg^*T=\ol{\pr}^*_bTg^*\ \ \mbox{on $\Omega^{0,q}(X)$},\ \ \forall g\in G,\\
&Tg^*\Box^{(q)}_b=g^*T\Box^{(q)}_b=\Box^{(q)}_bg^*T=\Box^{(q)}_bTg^*\ \ \mbox{on $\Omega^{0,q}(X)$},\ \ \forall g\in G.
\end{split}\]

Let $L^2_{(0,q),m}(X)_k$ be
the completion of $\Omega^{0,q}_{m}(X)_k$ with respect to $(\,\cdot\,|\,\cdot\,)$. 
We write $L^2_{m}(X)_k:=L^2_{(0,0),m}(X)_k$. Put 
\[
H^q_{b, m}(X)_k:=({\rm Ker\,}\Box^{(q)}_{b})\bigcap L^2_{(0,q),m}(X)_k.
\]
The $m$-th $G$-equivariant Szeg\H{o} projection is the orthogonal projection 
\[S^{(q)}_{k,m}:L^2_{(0,q)}(X)\To ({\rm Ker\,}\Box^{(q)}_{b})\bigcap L^2_{(0,q),m}(X)_k\]
with respect to $(\,\cdot\,|\,\cdot\,)$. Let $S^{(q)}_{k,m}(x,y)\in C^\infty(X\times X,T^{*0,q}X\boxtimes(T^{*0,q}X)^*)$ be the distribution kernel of $S^{(q)}_{k,m}$. 
Then
\begin{equation}\label{e-gue170111cr}
S^{(q)}_{k,m}(x,y)=\frac{1}{2\pi}\int^{\pi}_{-\pi}S^{(q)}_k(x,e^{i\theta}\circ y)e^{im\theta}d\theta.
\end{equation}
The goal of this section is to study the asymptotics of $S^{(q)}_{k,m}$ as $m\To+\infty$. 

From Theorem~\ref{t-gue170110}, \eqref{e-gue170111cr} and by using integration by parts several times, we get 

\begin{theorem}\label{t-gue170111w}
	Let $D\subset X$ be an open set with $D\bigcap\mu^{-1}(0)=\emptyset$. Then, 
	\[S^{(q)}_{k,m}=O(m^{-\infty})\ \ \mbox{on $D$}. \]
\end{theorem}

We now study $S^{(q)}_{k,m}$ near $\mu^{-1}(0)$. We can repeat the proof of Theorem~\ref{t-gue161202} with minor change and get 

\begin{theorem}\label{t-gue161202z}
	Let $p\in\mu^{-1}(0)$. There exist local coordinates $v=(v_1,\ldots,v_d)$ of $G$ defined in  a neighborhood $V$ of $e_0$ with $v(e_0)=(0,\ldots,0)$, local coordinates $x=(x_1,\ldots,x_{2n+1})$ of $X$ defined in a neighborhood $U=U_1\times(\hat U_2\times]-2\delta,2\delta[)$ of $p$ with $0\leftrightarrow p$, where $U_1\subset\Real^d$ is an open set of $0\in\Real^d$,  $\hat U_2\subset\Real^{2n-d}$ is an open set of $0\in\Real^{2n-d} $, $\delta>0$, and a smooth function $\gamma=(\gamma_1,\ldots,\gamma_d)\in C^\infty(\hat U_2\times]-2\delta,2\delta[,U_1)$ with $\gamma(0)=0\in\Real^d$  such that
	\begin{equation}\label{e-gue161202z}
	\begin{split}
	&(v_1,\ldots,v_d)\circ (\gamma(x_{d+1},\ldots,x_{2n+1}),x_{d+1},\ldots,x_{2n+1})\\
	&=(v_1+\gamma_1(x_{d+1},\ldots,x_{2n+1}),\ldots,v_d+\gamma_d(x_{d+1},\ldots,x_{2n+1}),x_{d+1},\ldots,x_{2n+1}),\\
	&\forall (v_1,\ldots,v_d)\in V,\ \ \forall (x_{d+1},\ldots,x_{2n+1})\in\hat U_2\times]-2\delta,2\delta[,
	\end{split}
	\end{equation}
	\begin{equation}\label{e-gue161206z}
	\begin{split}
	&T=-\frac{\pr}{\pr x_{2n+1}},\\
	&\underline{\mathfrak{g}}={\rm span\,}\set{\frac{\pr}{\pr x_1},\ldots,\frac{\pr}{\pr x_d}},\\
	&\mu^{-1}(0)\bigcap U=\set{x_{d+1}=\cdots=x_{2d}=0},\\
	&\mbox{On $\mu^{-1}(0)\bigcap U$, we have $J(\frac{\pr}{\pr x_j})=\frac{\pr}{\pr x_{d+j}}+a_j(x)\frac{\pr}{\pr x_{2n+1}}$, $j=1,2,\ldots,d$}, 
	\end{split}
	\end{equation}
	where $a_j(x)$ is a smooth function on $\mu^{-1}(0)\bigcap U$, independent of $x_1,\ldots,x_{2d}$, $x_{2n+1}$ and $a_j(0)=0$, $j=1,\ldots,d$, 
	\begin{equation}\label{e-gue161202Iz}
	\begin{split}
	&T^{1,0}_pX={\rm span\,}\set{Z_1,\ldots,Z_n},\\
	&Z_j=\frac{1}{2}(\frac{\pr}{\pr x_j}-i\frac{\pr}{\pr x_{d+j}})(p),\ \ j=1,\ldots,d,\\
	&Z_j=\frac{1}{2}(\frac{\pr}{\pr x_{2j-1}}-i\frac{\pr}{\pr x_{2j}})(p),\ \ j=d+1,\ldots,n,\\
	&\langle\,Z_j\,|\,Z_k\,\rangle=\delta_{j,k},\ \ j,k=1,2,\ldots,n,\\
	&\mathcal{L}_p(Z_j, \ol Z_k)=\mu_j\delta_{j,k},\ \ j,k=1,2,\ldots,n
	\end{split}
	\end{equation}
	and 
	\begin{equation}\label{e-gue161219z}
	\begin{split}
	\omega_0(x)&=(1+O(\abs{x}))dx_{2n+1}+\sum^d_{j=1}4\mu_jx_{d+j}dx_j\\
	&\quad+\sum^n_{j=d+1}2\mu_jx_{2j}dx_{2j-1}-\sum^n_{j=d+1}2\mu_jx_{2j-1}dx_{2j}+O(\abs{x}^2).
	\end{split}
	\end{equation}
\end{theorem}

From Theorem~\ref{t-gue170112}, we get 

\begin{theorem}\label{t-gue170112I}
	Assume that $q\notin\set{n_-,n_+}$. Then, $S^{(q)}_{k,m}=O(m^{-\infty})$ on $X$. 
\end{theorem}

\begin{proof}[Proof of Theorem \ref{t-gue170128I}]
It suffices to show the cases when $q=n_-$ and $q=n_+\neq n_-$.
Assume that $q=n_-$. It is well-known ~\cite[Theorem 1.12]{HM14} that when $X$ admits a transversal $S^1$ action, then  $\Box^{(q)}_b$ has $L^2$ closed range. 
Fix $p\in\mu^{-1}(0)$. Let $N_p=\{g\in G: g\circ p=p\}=\{g_1=e_0,g_2...,g_r\}$.
Let $v=(v_1,\ldots,v_d)$ and $x=(x_1,\ldots,x_{2n+1})$ be the local coordinates of $G$ and $X$ as in Theorem~\ref{t-gue161202z} and let $U$ and $V$ be open sets as in Theorem~\ref{t-gue161202z}. We take $U$ small enough so that there is a constant $c>0$ such that 
\begin{equation}\label{e-gue170117tI}
d(e^{i\theta}\circ g\circ x,y)\geq c,\ \ \forall (x,y)\in U\times U,\ \ \forall g\in G, \theta\in[-\pi,-\delta]\bigcup[\delta,\pi],
\end{equation}
where $\delta>0$ is as in Theorem~\ref{t-gue161202z}. 
We repeat the same procedure in \cite[Section 4]{HH} as follows.
\begin{equation}\label{e-gue170117}
\begin{split}
S^{(q)}_{k,m}(x,y)&=\frac{1}{2\pi}\int^{\pi}_{-\pi}S^{(q)}_k(x,e^{i\theta}\circ y)e^{im\theta}d\theta=\frac{1}{2\pi}\int^{\pi}_{-\pi}e^{-imx_{2n+1}+imy_{2n+1}}S^{(q)}_k(\mathring{x},e^{i\theta}\circ\mathring{y})e^{im\theta}d\theta\\
&=I+II,\\
&I=\frac{1}{2\pi}\int^{\pi}_{-\pi}e^{-imx_{2n+1}+imy_{2n+1}}\chi(\theta)S^{(q)}_k(\mathring{x},e^{i\theta}\circ\mathring{y})e^{im\theta}d\theta,\\
&II=\frac{1}{2\pi}\int^{\pi}_{-\pi}e^{-imx_{2n+1}+imy_{2n+1}}(1-\chi(\theta))S^{(q)}_k(\mathring{x},e^{i\theta}\circ\mathring{y})e^{im\theta}d\theta,
\end{split}
\end{equation}
where $\mathring{x}=(x_1,\ldots,x_{2n},0)\in U$, $\mathring{y}=(y_1,\ldots,y_{2n},0)\in U$, $\chi\in C^\infty_0(]-2\delta,2\delta[)$, $\chi=1$ on $[-\delta, \delta]$. 
It is easy to check that
\begin{equation}\label{e-gue170117pI}
II=O(m^{-\infty}). 
\end{equation}
For $I$, we have 
\begin{equation}\label{e-gue170117pII}
\begin{split}
&I=I_0+I_1,\\
&I_0=\frac{1}{2\pi}\sum_{\alpha=1}^{r}\int^\infty_0\int^{\pi}_{-\pi}
e^{-imx_{2n+1}+imy_{2n+1}}\chi(\theta)
e^{i(-\theta+\hat\Phi_{k,-}(g_\alpha\circ\mathring{x},\mathring{y}))t+im\theta}
a_{k,\alpha,-}(\mathring{x},
(\mathring{y},-\theta),t)dtd\theta,\\
&I_1=\frac{1}{2\pi}\sum_{\alpha=1}^{r}\int^\infty_0\int^{\pi}_{-\pi}
e^{-imx_{2n+1}+imy_{2n+1}}\chi(\theta)
e^{i(\theta+\hat\Phi_{k,+}(g_\alpha\circ\mathring{x},\mathring{y}))t+im\theta}
a_{k,\alpha,+}(\mathring{x},(\mathring{y},-\theta),t)dtd\theta.
\end{split}
\end{equation}
From $\frac{\pr}{\pr\theta}\Bigr(i(\theta+\hat\Phi_{k,+}(g_\alpha\circ\mathring{x},\mathring{y}))t+im\theta\Bigr)\neq0$, we can integrate by parts with respect to $\theta$ several times and deduce that 
\begin{equation}\label{e-gue170117pIII}
I_1=O(m^{-\infty}).
\end{equation}
For $I_0$, we have 
\begin{equation}\label{e-gue170117pIV}
I_0=\frac{1}{2\pi}\sum_{\alpha=1}^{r}\int^\infty_0\int^{\pi}_{-\pi}
e^{-imx_{2n+1}+imy_{2n+1}}\chi(\theta)
e^{im(-\theta t+\hat\Phi_{k,-}(g_\alpha\circ\mathring{x},\mathring{y})t+\theta)}
ma_{k,\alpha,-}(\mathring{x},
(\mathring{y},-\theta),mt)dtd\theta.
\end{equation}
We apply the complex stationary phase formula of Melin-Sj\"ostrand~\cite[Theorem 2.3]{MS74} to carry the $dtd\theta$ integration in \eqref{e-gue170117pIV}. The calculation is similar as in the proof of Theorem 3.17 in~\cite{HM14a}. Then 
\begin{equation}\label{e-gue170117pV}
\begin{split}
&I_0\equiv\sum_{\alpha=1}^{r}e^{im\Psi_k(g_\alpha\circ x,y)}b_{k,\alpha}(x,y,m),\\
&\Psi_k(x,y)=\hat\Phi_{k,-}(\mathring{x},\mathring{y})-x_{2n+1}+y_{2n+1},\\
&b_{k,\alpha}(x,y,m)\in S^{n-\frac{d}{2}}_{{\rm loc\,}}(1; U\times U, T^{*0,q}X\boxtimes(T^{*0,q}X)^*),\\
&\mbox{$b_{k,\alpha}(x,y,m)\sim\sum^\infty_{j=0}m^{n-\frac{d}{2}-j}
	b_{k,\alpha}^j(x,y)$ in $S^{n-\frac{d}{2}}_{{\rm loc\,}}(1; U\times U, T^{*0,q}X\boxtimes(T^{*0,q}X)^*)$},\\
&b_{k,\alpha}^j(x,y)\in C^\infty(U\times U, T^{*0,q}X\boxtimes(T^{*0,q}X)^*),\ \ j=0,1,2,\ldots,
\end{split}
\end{equation}

Assume that $q=n_+\neq n_-$. If $m\To-\infty$, then the expansion for $S^{(q)}_{k,m}(x,y)$ as  $m\To-\infty$ is similar to $q=n_-$ case. When $m\To+\infty$, we can repeat the method above with minor change and deduce that $S^{(q)}_{k,m}(x,y)=O(m^{-\infty})$ on $X$. In particular, it follows from the argument in \cite[Section 4]{HH} with minor modification, if $G \times S^1$ acts freely near $\mu^{-1}(0)$, then
	\[
	b_{k}^0(x,x)=2^{d-1}\frac{d^2_k}{V_{{\rm eff\,}}(x)}\pi^{-n-1+\frac{d}{2}}\abs{\det R_x}^{-\frac{1}{2}}\abs{\det\mathcal{L}_{x}}\tau_{x,n_-},\ \ \forall x\in\mu^{-1}(0),
\]
where $\tau_{x,n_-}$ is given by \eqref{tau140530}. The proof is completed. 
\end{proof}

\bibliographystyle{plain}

\begin{thebibliography}{99}
	\bibitem{BouSj76}
	L.~Boutet~de Monvel and J.~Sj{\"o}strand, \emph{Sur la singularit{\'e} des
		noyaux de {B}ergman et de {S}zeg{\"o}}, Ast{\'e}risque, \textbf{34--35} (1976), 123--164, MR0590106, Zbl 0344.32010.
	
	\bibitem{BG81} L. Boutet de Monvel and V. Guillemin,
	\emph{The spectral theory of Toeplitz operators},
	Ann. of Math. Stud.,
	vol 99, Princeton Univ. Press, Princeton, NJ, 1981.
	
	
	\bibitem{Ca99} D. Catlin, {\em The Bergman kernel and a theorem of Tian}, 
	in {\em Analysis and geometry in several complex variables (Katata, 1997)}, 1--23, 
	Trends Math., Birkh\"auser, Boston, 1999,  MR1699887, Zbl 0941.32002.
	
	
	\bibitem{Engl:02}
	M.\ Engli\v{s}, \emph{Weighted Bergman kernels and quantization}, Comm. Math. Phys. \textbf{227} 
	(2002), no.\ 2, 211--241, MR1903645, Zbl 1010.32002.
	
	
	\bibitem{FHH}
	K. Fritsch, H. Herrmann  and C.-Y.~Hsiao,
	\emph{$G$-equivariant embedding theorems for CR manifolds of high codimension},
	arXiv: 1810.09629.
	

	
	
	\bibitem{Gu89} 
	V.\ Guillemin, \emph{Star products on compact pre-quantizable symplectic manifolds},
	Lett. Math. Phys. \textbf{35} (1995), no. 1, 85--89

	\bibitem{HHL}
	H. Herrmann, C.-Y. Hsiao and X. Li, {\it Szeg\H{o} kernels and equivariant embedding theorems for CR manifolds}, arXiv: 1710.04910. 
	
	\bibitem{Hor03}
	L.~H\"{o}rmander,
	\emph{The analysis of linear partial differential operators. {I}}, Classics in Mathematics, Springer-Verlag, Berlin,
	2003.
	
	
	\bibitem{Hsiao08}
	C.-Y.~Hsiao,
	\emph{Projections in several complex variables}, M\'em. Soc. Math. France, Nouv. S\'er. \textbf{123} (2010), 131 p, MR2780123, Zbl 1229.32002.
	
	\bibitem{HH}
	C.-Y.~Hsiao and R.-T. Huang,
	\emph{$G$-invariant Szeg\H{o} kernel asymptotics and CR reduction},
	arXiv: 1702.05012.

	\bibitem{HMM}
	C.-Y. Hsiao, X. Ma and G. Marinescu, {\it Geometric quantization on CR manifolds}, arXiv: 1906.05627. 
	
	\bibitem {HM14a} C.-Y. Hsiao and G. Marinescu, {\it Asymptotics of spectral function of lower energy forms and Bergman kernel of semi-positive and big line bundles}, Comm. Anal. Geom. \textbf{22}(1) (2014), 1-108.

	
	\bibitem{HM14}
	C.-Y.~Hsiao and G.~Marinescu,
	\emph{On the singularities of the Szeg\H{o} projections on lower energy forms},
	J. Differential Geom. 107 (2017), no. 1, 83--155.
	
	


	\bibitem{Koh86}
	J.~J.~Kohn, \emph{The range of the tangential {C}auchy-{R}iemann operator}, 
	Duke Math.\ J.\ \textbf{53} (1986), No.\ 2, 307--562. 
	
	
	\bibitem{MM06}
	X.~Ma and G.~Marinescu, \emph{The first coefficients of the asymptotic expansion of the {Bergman} 
		kernel of the $spin^c$ {Dirac} operator}, Internat.\ J.\ Math.\ \textbf{17} (2006), no.~6, 737--759. 
	
	\bibitem{MM} X. Ma and G. Marinescu, \emph{Holomorphic Morse inequalities and Bergman kernels}, Progress in Mathematics, \textbf{254}, Birkh\"{a}user Verlag, Basel, (2007).

	
	\bibitem{MM08a}
	X.~Ma and G.~Marinescu, \emph{Generalized {B}ergman kernels 
		on symplectic manifolds}, 
	Adv.\ Math.\ \textbf{217} (2008), no.~4, 1756--1815.
	
	\bibitem{Ma10}  X.~Ma, 
	\emph{Geometric quantization on K{\"a}hler and symplectic manifolds},
	International {C}ongress of {M}athematicians, 
	vol.\ II, Hyderabad, India, August 19-27 (2010), 785--810.
	
       \bibitem{MZI} X. Ma and W. Zhang,\emph{Geometric quantization for proper moment maps: the Vergne conjecture}, Acta Math., \textbf{212},  (2014), no. 1, 11--57. 

	
	\bibitem{MS74}
	A.~Melin and J.~Sj\"{o}strand,
	\emph{Fourier integral operators with complex-valued phase functions},
	Springer Lecture Notes in Math., \textbf{459}, (1975), 120--223.

	
	\bibitem{Ros:65}
	H.~Rossi, \emph{Attaching analytic spaces to an analytic space along a
		pseudoconcave boundary}, Proc. Conf. Complex. Manifolds (Minneapolis),
	Springer--Verlag, New York, 1965, pp.~242--256. 

	\bibitem{Shen} W.-C. Shen, \emph{Asymptotics of torus equivariant Szeg\H{o} kernel on a compact CR manifold}, 
	Bull. Inst. Math. Acad. Sin. (N.S.) {\bf 14} (2019), no.3, 331-383.
	
	\bibitem{ShZ99} B. Shiffman and S. Zelditch, {\em Distribution of zeros of random and 
		quantum chaotic sections of positive line bundles}, Comm. Math. Phys. {\bf 200} (1999), 661--683.
	
	

	\bibitem{Ta} M.E. Taylor, \emph{Fourier series on compact Lie groups}, 
	Proc. Amer. Math. Soc. {\bf 19} (1968), no.5, 1103-1105.
	
	
	\bibitem{Zelditch98} S. Zelditch,
	\emph{Szeg\"{o} kernels and a theorem of Tian},
	Int. Math. Res. Not. \textbf{6} (1998), 317--331.
	
	
\end{thebibliography}

\end{document}